\documentclass[12pt]{amsart}
\usepackage{amssymb,epsfig,amsmath,latexsym,amsthm}
\usepackage{graphicx}
\theoremstyle{plain}
\newtheorem{theorem}{Theorem}[section]
\newtheorem{lemma}[theorem]{Lemma}

\newtheorem{proposition}[theorem]{Proposition}
\newtheorem{corollary}[theorem]{Corollary}
\newtheorem{question}[theorem]{Question}
\newtheorem{questions}[theorem]{Questions}
\newtheorem{conjecture}[theorem]{Conjecture}

\theoremstyle{definition}
\newtheorem{definition}[theorem]{Definition}
\newtheorem{definition/construction}[theorem]{Definition/Construction}
\newtheorem{construction}[theorem]{Construction}
\newtheorem{remark}[theorem]{Remark}

\newtheorem{acknowledgements}[theorem]{Acknowledgements}

\newtheorem{remarks}[theorem]{Remarks}

\DeclareMathOperator{\Diff}{Diff}

\DeclareMathOperator{\id}{id}

\DeclareMathOperator{\rel}{rel}

\DeclareMathOperator{\inte}{int}

\newcommand{\BC}{\mathbb C}

\newcommand{\BP}{\mathbb P}

\newcommand{\BR}{\mathbb R}
\newcommand{\BZ}{\mathbb Z}
\newcommand{\cptwo}{\BC\BP^2}
\newcommand{\cpone}{\BC\BP^1}
\newcommand{\stwostwo}{S^2\times S^2}
\newcommand{\stwosone}{S^2\times S^1}
\newcommand{\Rs}{R^{std}}
\newcommand{\Gs}{G^{std}}

\newcommand{\mA}{\mathcal{A}}

\newcommand{\mF}{\mathcal{F}}

\newcommand{\mH}{\mathcal{H}}

\newcommand{\mR}{\mathcal{R}}

\usepackage{epsfig,color}

\begin{document}

\title{The 4-dimensional Light Bulb Theorem}

\author{David Gabai}\address{Department of Mathematics\\Princeton
University\\Princeton, NJ 08544}

\thanks{Version 3.05, June 8, 2020\\Partially supported by NSF grants DMS-1006553, 1607374}
 
 \email{gabai@math.princeton.edu}

\begin{abstract} For embedded 2-spheres in a 4-manifold sharing the same embedded transverse sphere homotopy implies isotopy, provided the ambient 4-manifold has no $\BZ_2$-torsion in the fundamental group.  This gives a generalization of the classical light bulb trick to 4-dimensions, the uniqueness of spanning discs for a simple closed curve in $S^4$ and  $\pi_0(\Diff_0(S^2\times D^2)/\Diff_0(B^4))=1$.  In manifolds with $\BZ_2$-torsion, one surface can be put into a normal form relative to the other.  
\end{abstract}

\maketitle

\setcounter{section}{0}

\section{Introduction}\label{S0}  In his seminal work on immersions \cite{Sm1} Steven Smale classified regular homotopy classes of immersions of 2-spheres into Euclidean space and more generally into orientable smooth manifolds.  In \cite{Sm2} he gave the regular homotopy classification of immersed spheres in $\BR^n$ and asked:

\begin{question} (\textrm{Smale}, P. 329 \cite{Sm2}) Develop an analogous theory for imbeddings. Presumably this will be quite hard. However, even partial results in this direction would be interesting.\end{question}

This paper works in the smooth category and addresses the question of isotopy of spheres in 4-manifolds.  In that context Smale's results \cite{Sm1} show that two embeddings are homotopic if and only if they are regularly homotopic. Given that 2-spheres can knot in 4-space, isotopy is a much more restrictive condition than homotopy.  Indeed, the author is aware of only one unconditional positive result and that was proved more than 50 years ago: A 2-sphere in a 4-manifold that bounds a 3-ball is isotopic to a standard inessential 2-sphere, \cite{Ce1} p. 231, \cite{Pa}.

Recall that a \emph{transverse sphere} $G$ to a surface $R$ in a 4-manifold is a sphere with trivial normal bundle that intersects R exactly once and transversely.  The following are the main results of this paper.  

\begin{theorem}\label{main}  Let $M$ be an orientable 4-manifold such that $\pi_1(M)$ has no 2-torsion.  Two embedded 2-spheres with common transverse sphere $G$ are homotopic if and only if they are ambiently isotopic.  If they coincide near $G$, then the isotopy can be chosen to fix a neighborhood of $G$ pointwise. \end{theorem}

For a fundamental group with 2-torsion, i.e. having a nontrivial element of order 2, the methods of this paper yield the following.

\begin{theorem}\label{2-torsion}  Let $M$ be an orientable 4-manifold and $R_1, R_0$ be homotopic embedded spheres which coincide near the common transverse sphere $G$.  Then $R_1$ can be put into a normal form with respect to $R_0$ via an isotopy fixing a neighborhood of $G$ pointwise.  (See Definition \ref{normal form} for the description.)  Here $R_1$ has double tubes representing elements $\{[\lambda_1], \cdots, [\lambda_n]\}$ where the $[\lambda_i]$'s are distinct nontrivial 2-torsion elements and $R_1=R_0$ if this set is empty.  Any finite set of distinct 2-torsion elements gives rise to a $R_1$ in normal form with double tubes representing this set and two such $R_1$'s are isotopic if the corresponding set of $[\lambda_i]$'s are equal.  \end{theorem}

\begin{remark} Question 10.14 of the first version of this paper asked about the necessity of the 2-torsion condition.  Subsequently, two groups independently established the necessity of that condition.   Hannah Schwartz \cite{Sch} first found an explicit 4-manifold with a pair of homotopic but not topologically isotopic spheres with a common transverse sphere.  Rob Schneiderman and Peter Teichner \cite{ST}  used the Freedman - Quinn obstruction of Theorem 10.5 \cite{FQ}  to give an obstruction to topological isotopy and applied it to specific examples. \footnote{Note added in proof: R. Schneiderman \& P. Teichner, \emph{Homotopy verses isotopy: Spheres with duals in 4-manifolds}, preprint, shows that Freedman - Quinn is the obstruction.} Theorem \ref{2-torsion} leaves open the possibility that distinct normal form surfaces are isotopic.\end{remark}

\vskip 8pt

Generalizations of these results to multiple pairs of spheres is given in \S 10.  This generalization implies that if $R_0, R_1$ are  embedded spheres with a common transverse sphere and $M_1\to M$ is a finite cover such that $\pi_1(M_1)$ has no 2-torsion, then the preimages of $R_1$ are simultaneously isotopic to the preimages of $R_0$, though perhaps not equivariantly.    
\vskip 8pt
Here are some applications.

\begin{theorem}  A properly embedded disc in $S^2\times D^2$ is properly isotopic to a fiber if and only if its boundary is isotopically standard.\end{theorem}

\begin{theorem}  Two properly embedded discs $D_0$ and $D_1$ in $S^2\times D^2$ that coincide near their standard boundaries are properly isotopic rel boundary if and only if they are homologous in  $H_2(S^2\times D^2, \partial D_0)$.\end{theorem}

Let $\Diff_0(X)$ denote the group of diffeomorphisms of the compact manifold $X$ that are properly homotopic to the identity.

\begin{corollary}  $\pi_0(\Diff_0(S^2\times D^2)/\Diff_0(B^4))=1$.\end{corollary}

\begin{remark}  In words, modulo diffeomorphisms of the 4-ball, homotopy implies isotopy for diffeomorphisms of $S^2\times D^2$. \end{remark}

The classical \emph{light bulb} theorem states that a knot in $S^2\times S^1$ that intersects a $S^2\times y$ transversely and in exactly one point is isotopic to the standard vertical curve, i.e. a  $x\times S^1$.    The next result is the 4-dimensional version.

\begin{theorem}  (\textrm{4D-Lightbulb Theorem})\label{light bulb}  If $R$ is an embedded 2-sphere in $S^2\times S^2$, homologous to $x_0\times S^2$,  that intersects $S^2\times y_0$ transversely and only at the point $(x_0, y_0)$, then $R$ is isotopic to $x_0\times S^2$ 
via an isotopy fixing $S^2\times y_0$ pointwise.\end{theorem}

In 1985, under the above hypotheses, Litherland \cite{Li} proved that there exists a diffeomorphism \emph{pseudo-isotopic} to the identity that takes $R$ to $x_0\times S^2$ and  proved the full light bulb theorem for smooth $m$-spheres in $S^2\times S^m$ for $m>2$.   (There is an additional necessary condition in that case.)  Another version of the light bulb theorem was proven in 1986 by Marumoto \cite{Ma}.  He showed that two locally flat PL $m$-discs in an $n$-sphere, $n>m$ with the same boundary are topologically isotopic rel boundary.    Here we prove that theorem for discs in $S^4$ in the smooth isotopy category.

\begin{theorem} (Uniqueness of Spanning Discs) If $D_0$ and $D_1$ are  discs in $S^4 $ such that $\partial D_0 = \partial D_1=\gamma$, then there exists an isotopy of $S^4$ taking $D_0$ to $D_1$ that fixes $\gamma$ pointwise.\end{theorem}

\begin{remark}  The analogous result for 1-discs in $S^4$ is well known using general position.  The result for 3-discs in $S^4$ implies the smooth 4D-Schoenflies conjecture.\end{remark}

This paper gives two proofs of the 4D-Light Bulb Theorem.  The first proof has two steps.  First we give a direct argument showing that $R$ is isotopic to a \emph{vertical sphere}, i.e. viewing $S^2\times S^2$  as $S^2\times S^1\times [-\infty,\infty]$ where each $z\times S^1\times \infty$ and each $z\times S^1\times -\infty$ is identified with a point, then after isotopy $R$ is transverse to each $S^2\times S^1\times t$ and intersects each such space in a single component.  This involves an analogue of the normal form theorem of \cite{KSS} and repeated use of $S^2\times 0$ as a transverse 2-sphere.  The second step invokes Hatcher's \cite{Ha} theorem (the Smale conjecture: $\Diff^+(S^3)\simeq SO(4)$) to straighten out these intersections.  

The proof of Theorem \ref{main}, and hence a somewhat different one for $S^2\times S^2$ makes use of Smale's results on regular homotopy of 2-spheres in 4-manifolds \cite{Sm1}.   We show that if $R_0$ is homotopic to $R_1$ and both are embedded surfaces, then the homotopy from $R_0$ to $R_1$ is \emph{shadowed by tubed surfaces}, i.e. there is an isotopy taking $R_0$ to something that looks like $R_1$ embroidered with a complicated system of tubes together with parallel copies of the transverse sphere. Through various geometric arguments we show that these tubes can be reorganized and eventually isotoped away.  The proof formally relies on the first proof of the Light Bulb theorem at the very last step, though we outline how to eliminate the dependence in Remark \ref{dependence}.  The proof uses the fact that $R_0$ is a 2-sphere.  The $\BZ_2$-condition is used in Proposition \ref{double to single}.

Both arguments make use of the 4D-Light Bulb Lemma, which is the direct analogue of the 3D-version where one can do  a  crossing change using the transverse sphere. 

More is known in other settings.  In the topological category a locally flat 2-sphere in $S^4$ is topologically equivalent to the trivial 2-knot if and only if its complement has fundamental group $\BZ$ \cite{Fr}, \cite{FQ}.  There are topologically isotopic smooth 2-spheres in 4-manifolds that are not smoothly isotopic, yet become smoothly isotopic after a stabilization with a single $\stwostwo$ \cite{AKMR}, \cite{Ak}.

\vskip 8 pt
The paper is organized as follows.  \S2 recalls some classical uses of transverse spheres and proves the Light Bulb Lemma.  The Light Bulb theorem is proven in \S3.  Basic facts about regular homotopy are recalled in \S4.  The definition of tubed surface, basic operations on tubed surfaces, the notion of shadowing a homotopy by basic operations and normal form for a surface are given in \S5.  The reader is cautioned that \emph{tubes} are used in two contexts here; as tubes that follow curves lying in the surface and as tubes that follow arcs with endpoints in the surface. The latter  fall into two types; \emph{single} and \emph{double} tubes.   In \S6 it is shown how to transform certain pairs of double tubes into pairs of single tubes. If there is no $\BZ_2$-torsion, then in the end all but at most one of the double tubes remains and that one is homotopically inessential.   If $\pi_1(M)$ has 2-torsion, then there may be additional double tubes representing \emph{distinct} 2-torsion elements of  $\pi_1(M)$.    A crossing change lemma is proven in \S7 enabling distinct tubes following curves in the surface to be disentangled.  In \S8 the proof of Theorems \ref{main} and  \ref{2-torsion} is completed.  An extension to higher genus surfaces is given in \S9.  In particular it is shown that a closed oriented surface in $\stwostwo$ homologous to $0\times S^2$, that intersects $S^2\times 0$ transversely in one point, is isotopically standard.  Applications and questions are given in \S 10.

\begin{acknowledgements}  This work was carried out in part while attending many annual Dublin topology workshops and during preceding visits to Trinity College, Dublin.  We thank Martin Bridgeman for making this possible.  Also, this work was partially carried out while the author was a member of the Institute for Advanced Study in 2015/16. Revisions were partially written while visting OIST and SUST.  We thank all these institutions for their hospitality.  We thank Bob Edwards for his many constructive comments and for his interest over many years while this project developed. See \cite{Ed}.   We thank Abby Thompson for asking about higher genus surfaces and  Remark \ref{thompson}.  We thank Hannah Schwartz for many conversations about the first version of this paper and her work on the  2-torsion problem.  We thank Maggie Miller for correcting a reference.  We thank Rob Schneiderman and Peter Teichner for their interest in this work and informing us of their use of the Freedman - Quinn obstruction.  We are grateful to the referees for their many comments, suggestions and constructive criticisms.
\end{acknowledgements}

\section{the 4-dimensional light bulb lemma}

Unless said otherwise, all manifolds in this paper are smooth and orientable and immersions are self-transverse.

\begin{definition} A \emph{transverse sphere} $G$ to the immersed surface $R$ is a sphere with trivial normal bundle that intersects $R$ transversely in a single point.  \end{definition}

All transverse spheres in this paper are embedded.  The following is well known.  We give the proof as a warm up to the light bulb lemma. 

\begin{lemma} \label{pi injectivity} If $R$ is an immersed surface with embedded transverse sphere $G$ in the 4-manifold $M$, then the induced map $\pi_1(M\setminus R) \to \pi_1(M)$ is an isomorphism.  If $R$ is a sphere, then the induced maps  $\pi_1(M\setminus R\cup G)\to \pi_1(M)$  and $\pi_1(M\setminus G)\to \pi_1(M)$ are isomorphisms.\end{lemma}

\begin{proof}  Surjectivity is immediate by general position.  If $\gamma$ is a loop in $M\setminus R$ bounding the singular disc $D\subset M$, then after a small perturbation we can assume that $D$ is transverse to $R$. Tubing off intersections with copies of $G$ shows that the map is also injective.  Using transverse spheres to remove intersections goes back to Norman \cite{No}.  For the second and third cases, we can assume that $D$ is transverse to $R\cup G$.  First use $R$ to tube off intersections of $D$ with $G$.  This proves injectivity for the third case.  (If $R$ does not have a trivial normal bundle or is not embedded, then the resulting disc may have extra intersections with $R$.)  Tubing with $G$ eliminates all the $D\cap R$ intersections and so the induced map in the second case is also injective.  \end{proof}

The light bulb lemma basically says that in the presence of a transverse sphere one can do an ambient isotopy of a surface R as in shown in Figure 2.1, without  introducing any self intersections.  

\begin{figure}[ht]
\includegraphics[scale=0.60]{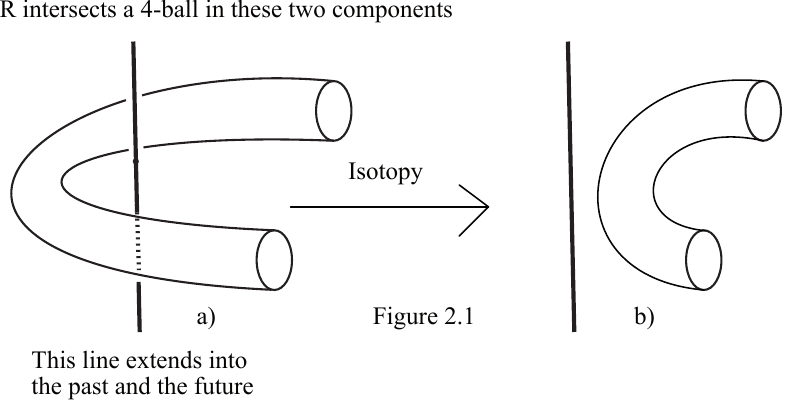}
\end{figure}

\begin{lemma}  \label{light bulb trick}(\textrm{4D-Light Bulb Lemma}) Let $R$ be an embedded surface with transverse sphere $G$ in the 4-manifold $M$ and let $z=R\cap G$.  Let $\alpha_0$ and $\alpha_1$ be two smooth compact arcs that coincide near their endpoints and bound the pinched embedded disc $E$ that is transverse to $R$ with $R\cap E=y$ and $E\cap G=\emptyset$.  See Figure 2.2 a).  Let $f_t$ be an ambient isotopy of $M$ taking $\alpha_0$ to $\alpha_1$ that corresponds to sweeping $\alpha_0$ across $E$.  Here $f_t$ is  fixed near $\partial \alpha_0$ and is supported in a small neighborhood of $E$.  Suppose that $N(\alpha_0)$ is parametrized as $B^3\times I$ and $R\cap N(E)=C\cup B$, where $C$ is the disc containing $y$ and $B\subset\inte(B^3)\times I$.  If $y$ and $z$ lie in the same  component of $R\setminus B$,  then $R$ is ambiently isotopic to $g(R)$ where $g|R\setminus B= \id$ and $g|B=f_1|B$.  The ambient isotopy fixes $G$ pointwise and the isotopy restricted to $R$ is supported in $B$.  

If $G$ has a non trivial normal bundle with even Euler class, then the conclusion holds except for the assertion that the ambient isotopy fixes $G$.  

If the Euler class is odd, then under the additional hypothesis that $B=L\times I\subset\inte(B^3)\times I$, where $L$ is an unlink in $\inte(B^3)$, the above conclusion holds with the additional modification $g(B)=f_1(B)$, where $g$ differs from $f_1$ by a Dehn twist along each of the tubes.   In general $g|N(\alpha_0)$ is the composition of the result of the standard isotopy taking $N(\alpha_0)$ to $N(\alpha_1)$, followed by the non trivial element of $\pi_1(SO(3))$ along $N(\alpha_1)$.
\end{lemma}

\begin{remarks}  i) After an initial isotopy of $R$ supported near $N(\alpha_0)$ we can assume that it is of the form $L\times I$ where $L$ is a link in $\inte(B^3)\times 0$.   

ii)  The hypothesis does not hold if $B$ separates $y$ from $z$ in $R$.

iii) Note that a right Dehn twist on a tube  is isotopic to a left Dehn twist on that component.  \end{remarks} 

\begin{proof}  Since $y$ lies in the same component of $z$ we can tube off $E$ with a copy of $G$ to obtain a disc $D$ that coincides with $E$ near $\partial E$ and $D\cap(R\cup G)=\emptyset$.  Since $G$ has a trivial normal bundle, there exists a framing of the normal bundle of $D$ that coincides with that of $E$ near $\partial E$.  See Figure 2.2.  Therefore, we can isotope $B$ to $f_1(B)$ by sweeping across $D$ rather than $E$. This isotopy is supported in a neighborhood of $D$ that is disjoint from a neighborhood of $G$.  

When $G$ has a nontrivial normal bundle with Euler class $n$, then $|D\cap G|=n$ and so the ambient isotopy taking $N(\alpha_0)$ to $N(\alpha_1)$ does not fix $G$. This isotopy is the composition of the standard one followed by $n$ full twists along $N(\alpha_1)$.  Since $\pi_1(SO(3))=\BZ_2$, the twisting can be isotopically undone when $n$ is even.  When $n$ is odd the twisting can be isotoped to a single full twist.  If the tubes in $B$ are unknotted and unlinked, then they can be isotoped so that $g(B)=f_1(B)$ where $g$ differs from $f_1$ by a Dehn twist along each of the tubes. \end{proof}

\begin{figure}[ht]
\includegraphics[scale=0.60]{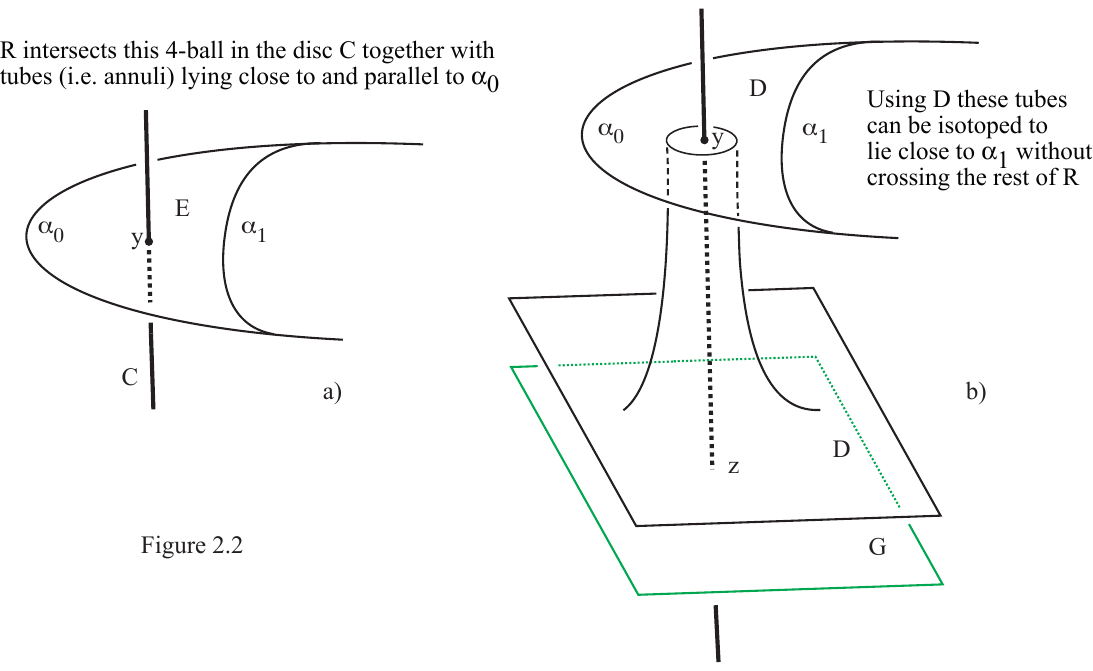}
\end{figure}

\section{The light bulb theorem for $S^2\times S^2$}

\begin{theorem} \label{4d lightbulb} If $R$ is a 2-sphere in $S^2\times S^2$, homologous to $x_0\times S^2$, transverse to $S^2\times y_0$, $R\cap (S^2\times y_0)=(x_0,y_0)$ and coincides with $x_0\times S^2$  near $(x_0,y_0)$, then $R$ is smoothly isotopic to $x_0\times S^2$ via an isotopy fixing  a neighborhood of $S^2\times y_0$ pointwise.\end{theorem}

\begin{definition}  A \emph{light bulb} in $\stwostwo$ is a smooth 2-sphere transverse to a $S^2\times y_0$ and intersects $S^2\times y_0$ in a single point.  View $S^2\times S^2$ as a quotient of $S^2\times (S^1\times [-\infty,\infty]) $  where each $x\times S^1\times -\infty$ and $x \times S^1\times \infty$ are identified with points and $y_0$ is identified with $(z_0, 0)\in S^1\times [-\infty,\infty]$.  We say that the light bulb $R$ is \emph{vertical} if it is transverse to each $S^2\times S^1\times u$, for $u\in [-\infty\times \infty]$.  Let $\Gs$ denote the sphere $S^2 \times z_0\times 0$  and $\Rs$ denote the sphere $x_0 \times S^1\times [-\infty,\infty]\subset \stwostwo$.    \end{definition}

To prove the light bulb theorem it suffices to assume that $R$ and $\Rs$ coincide in some neighborhood $U$ of $(x_0, z_0, 0)$.
\vskip 8pt
\noindent\emph{Step 1.} The light bulb $R$ is isotopic to a vertical light bulb by an isotopy fixing a neighborhood of $\Gs$ pointwise. 
\vskip 8pt

\noindent\emph{Step 1A.}  We can assume that $R$ coincides with $\Rs$ within $S^2\times (z_0-\epsilon, z_0+\epsilon)\times [-\infty, \infty]\cup (S^2\times S^1\times [-\infty, 10))\cup (S^2\times S^1\times (10,\infty]$.   

\vskip8pt
\noindent\emph{Proof.}  This follows from the fact that $R$ intersects a neighborhood of $S^2\times z_0\times 0$ as does $\Rs$ and a small regular neighborhood of $\Gs$ is naturally ambiently isotopic to $S^2\times (z_0-\epsilon, z_0+\epsilon) \times [-\infty,\infty]\cup (S^2\times S^1\times [-\infty,10))\cup S^2\times S^1\times (10,\infty])$.\qed
\vskip 8pt

From now on we will take $U$ to be the neighborhood of $(S^2, z_0, 0)$ given in the statement of Step 1A.  Note that $U$ is the complement of $S^2\times [z_0+\epsilon, z_0-\epsilon]\times [-10,10]$, where $S^1=[z_0+\epsilon, z_0-\epsilon]\cup (z_0-\epsilon, z_0+\epsilon)$.
\vskip 10pt

\noindent\emph{Step 1B.} Via an isotopy fixing $R\cap U$, $R$ can be isotoped to be  transverse to each $S^2\times S^1\times u$ except for $u=-9, -6, 6, 9$.  As $u$ increases, $p$ local minima (with respect to $u$) appear at $u=-9$, $p$ saddles appear at $u= -6$, $R\cap S^2\times S^1\times u$ is connected for $u\in (-6,6)$, $q$ saddles appear when $u=6$ and $q$ local maxima appear when $u=9$.  

\vskip 8pt

\noindent\emph{Proof.}  This is the analogy of the normal form of \cite{KSS} in our setting, stated in the smooth category.    See \cite{CKS} pages 6-8 for a modern proof of \cite{KSS}.  Swenton \cite{Sw} more or less shows uniqueness the normal form up to a certain set of moves.

What follows is a brief outline of the proof in our setting. In the usual manner $R$ can be isotoped so that it is transverse to each $\stwosone\times u$ except for $u=-9, 0, 9$ where local minima, saddles, local maxima respectively appear.  As in the above papers, up to smoothing of corners, the local minima (resp. maxima) correspond to the appearance of discs and the saddles correspond to the appearance of bands.  After further isotopy, as in Figure 1.3 \cite{CKS}, we can assume that the bands are disjoint from each other, so for $\delta$ small,  $R\cap S^2\times S^1\times \delta$ is the result of doing band sums to $R\cap S^2\times S^1\times -\delta$.    


If $p$ (resp. $q$) is the number of local minima (resp. maxima), then since $\chi(R)=2$ the total number of saddles is $p+q$.  Since $R$ is connected there exist $p$ bands such that the result of only doing these band sums to the curves corresponding to $R\cap S^2\times S^1\times -\delta$ yield a connected curve.  Push these bands to $\stwosone\times-6$ and push the remaining bands to $\stwosone\times 6$.  \qed
\vskip 8pt

In what follows we  let $C_u$ denote the \emph{core curve} i.e. the component of $R\cap \stwosone\times u$ which is transverse to $S^2\times z_0\times u$, $u\neq -6, 6$.  Define $C_{-6}=\lim_{t\to-6_-}C_t$.  We abuse terminology by calling a \emph{core curve} such  a curve $C$ without specifying $u$.   After  band sliding we can assume that all the bands at $u=-6$ have one end that attaches to the core curve.

In summary, up to smoothing corners, we can assume that $R\cap \stwosone\times [-10, -5]$ appears as follows.  For $u\in [-10,-9),\ R\cap \stwosone\times u$ 
is the standard core curve $x_0\times S^1\times u$.  At $u=-9$, discs $D_1, \cdots, D_p$ appear.  Let $c_1,\cdots, c_p$ denote their boundary 
curves. The surface $\ R\cap \stwosone\times (-9,-6)$ is the product $(C\cup c_1\cup\cdots\cup c_p)\times (-9,-6)$.  Here we again abuse notation 
by denoting a $c_i$ without specifying its $u$ level.  At $u=-6, p$ bands $b_1, \cdots, b_p$ appear where $b_i$ connects $C$ and $c_i$.  Again 
$R\cap \stwosone\times (-6,-5]$ is a product where each $u$ section is parallel to $R\cap \stwosone\times -6$ with the relative interiors of the bands 
removed. 

By a vertical isotopy push the bands $b_2, \cdots, b_p$ up to level $-5$ and the disc $D_1$ to level $-8$.   Let $\pi:\stwosone\times [-\infty,\infty]\to \stwosone$ be the projection.  To complete the proof of Step 1 we will show that after isotopy  $\pi(b_1)\cap \pi(D_1)\subset \pi(\partial D_1$).  It follows that  $b_1$ can be pushed to level $-8$ and its critical point can be cancelled with the one corresponding to $D_1$.  Step 1 then follows by induction and the usual turning upside down argument to cancel the saddles at $u=6$ with the maxima at $u=9$.

\vskip 8pt
\noindent\emph{Step 1C.} There exist pairwise disjoint discs $E_1, \cdots, E_p \subset S^2\times S^1\times -6 $ spanning $c_1, \cdots, c_p$ such that for all $i,\ \pi(\inte(E_i))\cap \pi(b_1)=\emptyset$ and $E_i\cap C\cup U=\emptyset$.  
\vskip 8pt
\noindent\emph{Proof.}  To start with, for $i= 1, \cdots, p$, let $E_i=D_i$.  A given $E_i$ projects to one intersecting  $\pi(b_1)$ in finitely many interior arcs.  View $b_1$ as a band starting at $c_1$ and sequentially hitting the various $E_i$'s before attaching to $C$.  Again we abuse notation by suppressing the fact that we should be talking about projections.   Starting at the last intersection of $b_1$ with an $E_i$, sequentially isotope the $E_i$'s to remove arcs of intersection at the cost of creating two points of intersection of an $E_i$ with $C$.  This type of argument was used in \cite{Gi} and \cite{Go}.  Next by following $C$ but avoiding the arc $b_1\cap C$ tube off these intersections with parallel copies of $S^2\times z_0$ to obtain the desired set of discs which we still call $E_1, \cdots E_p$.  See Figure 3.1\qed

\vskip 8pt
\begin{figure}[ht]
\includegraphics[scale=0.60]{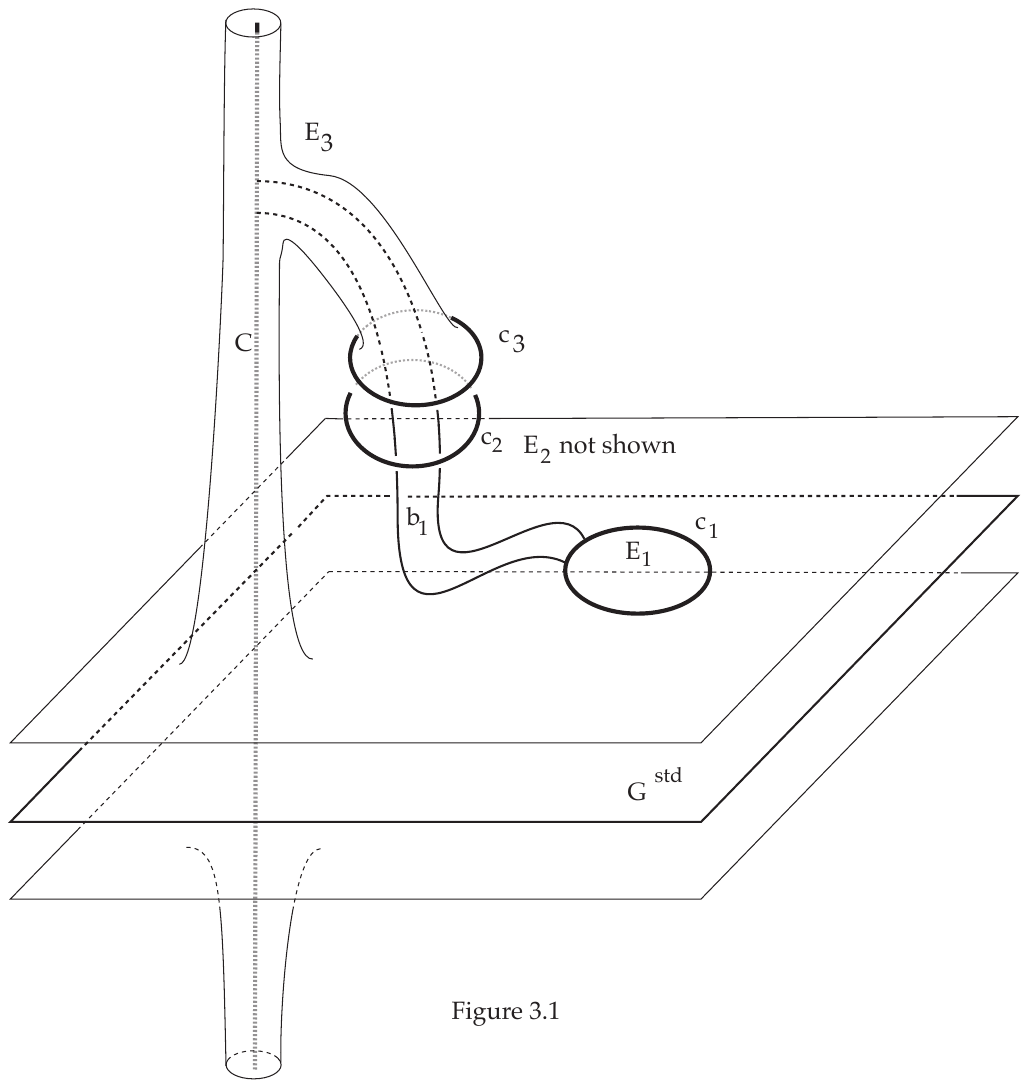}
\end{figure}

\begin{remark}\label{straight} For the purposes of visualization one can  ambiently isotope $R$ via level preserving isotopy supported in $S^2\times S^1\times [-9.5, -5.5]$ so that  the  discs $E_i $ become  small and round and $b_1$ becomes a straight band connecting $C$ and $c_1$, which is disjoint from the interior of the $E_i$'s.   Furthermore, up to rounding corners, possibly complicated discs $D_2, \cdots, D_p$ appear at level $-9$, a possibly complicated disc $D_1$ appears at level $-8$ and vertical annuli $c_1\times [-8,-6], c_2\times [-9,-6], \cdots, c_p\times [-9,-6]$ emanate from the $\partial D_i$'s.  At level $-6$, $R$ appears as in the first sentence of this remark.\end{remark}

\vskip 8pt

\noindent\emph{Step 1D.} We can assume that $\pi(D_1)\cap \pi(E_i)=\emptyset$ for $i>1$.

\vskip 8pt

\noindent\emph{Proof.}  Let $\pi_{-8}$ denote the projection of $S^2\times S^1\times -6$ to $S^2\times S^1\times -8$ fixing the first two factors.  By construction $\partial D_1=c_1$ and $D_1$ is disjoint from $C\cup U$ as well as the $c_i$'s for $i>1$.   Let $F_i$ denote $\pi_{-8}(E_i)$.  We show how to isotope $D_1$ off  $F_2$.   Let $\alpha_0\subset S^2\times S^1\times -8$ be an arc transverse to $\inte(F_2)$.   Then $N(\alpha_0)=D^2\times I\times I$ where $D^2\times I\times 1/2\subset S^2\times S^1\times -8$.  Next isotope $D_1$ so that $N(F_2)\cap D_1\subset (1/2 D^2)\times I\times 1/2$  and $D_1\cap (D^2\times I\times 1/2)$ is of the form $L\times I\times 1/2$, where $L$ is a union of circles.    Here we identify $\alpha_0$ with $t_0\times I\times 1/2$ where $t_0\notin 1/2 D^2$.  Let $\alpha_1\subset S^2\times S^1\times -8$ be an arc disjoint from $R\cup F_2\cup\cdots\cup F_p$ whose ends coincide with those of $\alpha_0$, such that $\alpha_0\cup\alpha_1$ bounds a pinched disc $E\subset S^2\times S^1\times -8$ that only intersects $R$ in a point of $\partial F_2$.  Using the fact that $R\setminus D_1$ is connected and  intersects $\Gs$ transversely once it follows from the Light Bulb Lemma \ref{light bulb trick} that $R$ is isotopic to the surface obtained by  isotoping $D_1 \cap N(\alpha_0)$ across $E$ into $N(\alpha_1)$ via an isotopy supported in $D_1\cap \inte(N(\alpha_0))$.   This isotopy moves $D_1$ off $F_2$ without introducing new intersections with other $F_i$'s.   See Figure 3.2.   Step 1D now follows by induction.   \qed
\vskip 8pt

\begin{figure}[ht]
\includegraphics[scale=0.60]{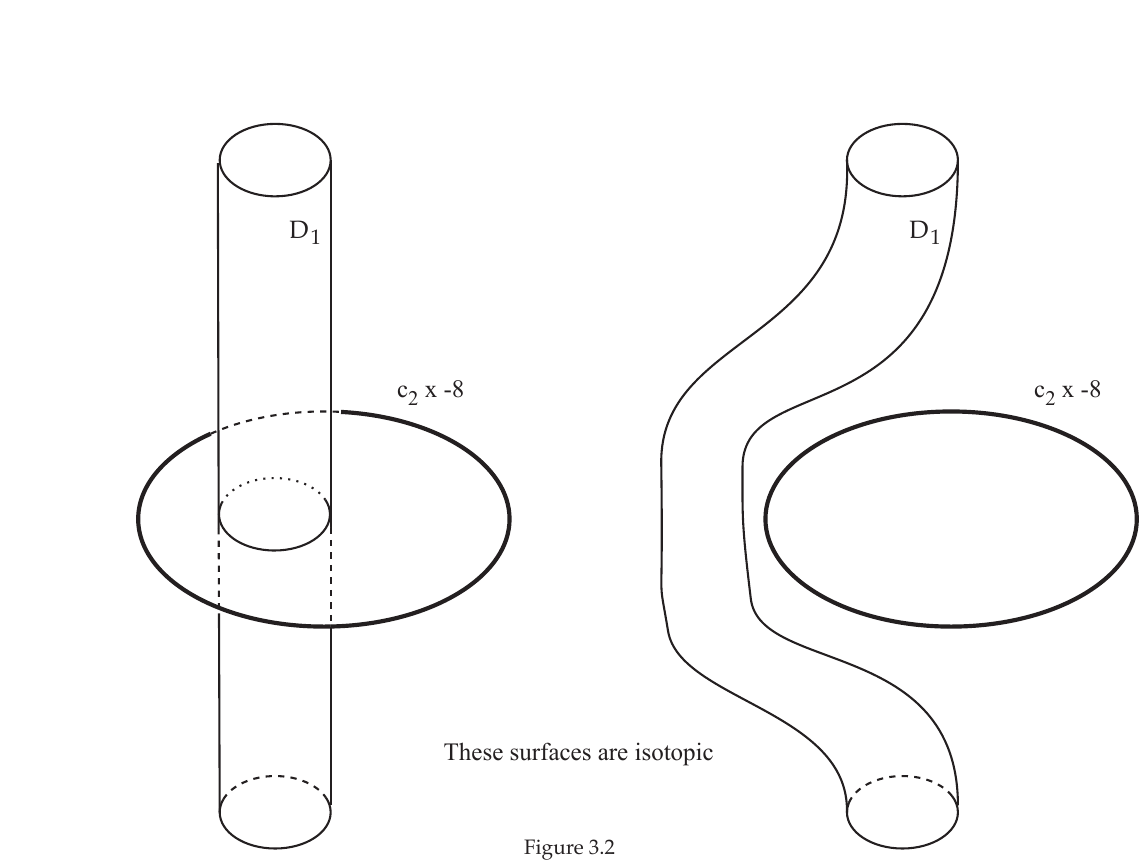}
\end{figure}

To summarize the situation at the moment:  At level $-9$ discs $D_2, \cdots, D_p$ appear, at level $-8$ disc $D_1$ appears,   vertical annuli $c_1\times [-8,-6], c_2\times [-9,-6], \cdots, c_p\times [-9,-6]$ emanate from the $\partial D_i$'s,  a band connects the core to $c_1\times -6$ disjoint from $\inte(E_1)$ and by Step 1D, for $i>1$, $\pi(E_i)\cap \pi(D_1)=\emptyset$ 

We can therefore isotope $D_2, \cdots, D_p$ and hence $\cup_{i>1}c_i\times [-9,-6]$ so that $c_2, \cdots, c_p  $ are \emph{far away} from $D_1, E_1$ and $b_1$.  This means that $\pi(c_2), \cdots, \pi(c_p)  $ lie in a 3-ball $B\subset \stwosone$ that intersects $C$ in a connected unknotted arc and $B$ is disjoint from $\pi(D_1), \pi(E_1)$ and $\pi(b_1)$.  
\vskip 8pt

\noindent\emph{Step 1E.}  Cancel the critical points corresponding to $D_1$ and $b_1$ without introducing new ones, thereby completing Step 1.
\vskip 8pt
\noindent\emph{Proof.}  Note that $(S^2\times S^1\times -6) \setminus (C\cup B\cup S^2\times z_0\times -6)$ is diffeomorpic to $\BR^3$.  Therefore, the discs $\pi(D_1)$ and $\pi(E_1)$ are isotopic rel $c_1$ via an isotopy disjoint from $\pi(C\cup B)\cup (S^2\times z_0)$ since spanning discs for the unknot in $\BR^3$ are unique up to isotopy.  After the corresponding isotopy of $D_1$, supported in $S^2\times S^1\times -8$  it follows from Remark \ref{straight} that $\pi(b_1)\cap \pi(D_1)\subset \pi(\partial D_1)$.  Therefore, $b_1$ can be pushed down to level $-8$, hence the critical points corresponding to $b_1$ and $D_1$ can be cancelled. \qed.

\vskip 10pt
\begin{remark} \label{content}  The content of Steps 1B-1E is that excess critical points of a boundary standard $I\times I\subset (S^2 \times I)\times [-10,10]$  can be cancelled rel $ \partial$.   Cancelling critical points of submanifolds of product manifolds, to the extent possible is a long sought after goal.  E.g. under suitable hypothesis in higher dimensions results have been obtained in \cite{Sh} and \cite{Pe}.  In dimension-4, Scharlemann \cite{Sc} showed that a smooth 2-sphere in $\BR^4$ with at most four critical points is smoothly isotopically standard.  \end{remark}

\begin{remark} \label{thompson} Abby Thompson pointed out that the above arguments work for higher genus surfaces to eliminate critical points of index 0 and 2. So if genus$(R)=g$, then $R$ can be isotoped so that  $2g$  bands appear at $u=-6$ and there are no other critical levels.  

From the point of view of $S^2\times -6$ these bands can be twisted and linked.  See \S 9.  \end{remark} 

\noindent\textbf{Step 2.} A vertical light bulb $R$ homologous to $\Rs$ that agrees with $\Rs$ near $\Gs$ is isotopic to $\Rs$ via an isotopy fixing a neighborhood of $\Gs$ pointwise.

\begin{remark}  It is easy to construct a vertical lightbulb homologous to $[\Rs] + n[S^2\times z_0\times 0]$ by first starting with $\Rs$, removing a neighborhood of $(x_0, z_1,0)$ and replacing it by one that sweeps across $S^2\times z_1$ $n$ times, while $u\in (-\epsilon ,\epsilon)$, where $z_1\neq z_0$.\end{remark}

\noindent\emph{Proof}  The proof of Step 1 shows that we can assume that $R$ coincides with  $\Rs$  away from $S^2 \times S^1\times [-10,10]$ and coincides with $\Rs$ near $S^2 \times z_0\times [-\infty,\infty]$.  Thus $R$ is standard outside a submanifold $W$ of the form $S^2 \times [0,1]\times [-10,10]$ and within $W$ corresponds to a smooth path of embedded smooth paths  $\rho_t:D^1\to S^2\times I$ for $t\in[-10,10]$, where $\rho_{-10}(D^1)=\rho_{10}(D^1) = (x_0, I)$ and $\rho_t$ is fixed near the endpoints of $D^1$. By identifying $D^1$ with $(x_0,  I)$ we can assume that $\rho_{-10}=\rho_{10}=\id$.  Note that $\Rs$ corresponds to the identity path.  

By the covering isotopy theorem, $\rho_t$, extends to a path $\phi_t\in \Diff(S^2\times I,\rel (\partial S^2 \times I))$ with $\phi_{-10}=\id$.    We first show that such a path can be chosen so that $\phi_{10}=\id$.  By uniqueness of regular neighborhoods we can first assume that restricted to some $D^2$ neighborhood  of $x_0$, in polar coordinates, $\phi_{10}(r,\theta,s)=(r, \theta + h(s) 2 \pi,s)$ for some $h:[0,1]\to \BR$ with $h(0)=0$.  Since $[R]=[\Rs]\in H_2(S^2\times S^2)$ it follows that $h(1)=0$ and hence after further isotopy, that $\phi_{10}|D^2\times I=\id$.  Since $S^2\times I \setminus (\inte(D^2)\times I)= B^3$, we can assume that $\phi_{10}=\id$, by \cite{Ce2} or \cite{Ha}.  

Thus $\phi_t$ is a closed loop in $\Diff(S^2\times I,\rel \partial (S^2 \times I))$ which by Hatcher \cite{Ha} is homotopically trivial since $\pi_1(\Omega(O(3))=\pi_2(O(3))=\pi_2(\BR(P^3))=0$.  Here we are using formulation (8) (see the appendix of \cite{Ha}) of Hatcher's  theorem which asserts that $\Diff(D^1 \times S^2 \rel \partial)$ is homotopy equivalent to $\Omega(O(3))$.  Restricting this homotopy to $\rho_t$ gives the desired isotopy of $\rho_t$ to id.  \qed


\begin{theorem}  \label{uniqueness} Let $D$ denote $x_0\times D^2\subset S^2\times D^2$.  A properly embedded disc $D_0\subset S^2\times D^2$ that coincides with $D$ near $\partial D$ is isotopic to $D$ rel $\partial D$ if and only if it is homologous to $D$ in  $H_2(S^2\times D^2, \partial D)$.\end{theorem}

\begin{proof}  Homologous is certainly a necessary condition.  Let $M=S^2\times D^2\cup d(S^2\times D^2)=S^2\times S^2$ be obtained by doubling $S^2\times D^2$ with $d(S^2\times D^2)$ denoting the other $S^2\times D^2$.   This $d(S^2\times D^2)$ can be viewed as a regular neighborhood $N(G)$ of $G=d(S^2\times 0)$.    Let  $R$ denote the sphere  $D_0\cup d(x_0\times D^2)$ and $\Rs$ denote $D\cup d(x_0\times D^2)$.  $G$ is a transverse sphere to the homologous spheres $R$ and $\Rs$.  By Theorem \ref{light bulb} there is an  isotopy of $M$ fixing $N(G)$ pointwise taking $R$ to $\Rs
$.  Restricting to $S^2\times D^2$ yields the desired isotopy.  \end{proof}

\begin{conjecture}  The space of light bulbs is not simply connected.\end{conjecture}

\section{Regular homotopy of embedded spheres in 4-manifolds}

The main result of this section is essentially Theorem D of \cite{Sm1}.  First recall \cite{Sp} that a smooth \emph{immersion} $f:M\to N$ is a smooth map of maximal rank at each $x\in M$.  A smooth \emph{regular homotopy} $F:M\times I\to N$ is a smooth map such that each $F_t$ is an immersion.

\begin{theorem} (Smale (1957))  Two smooth embedded spheres in an orientable 4-manifold are regularly homotopic if and only if they are homotopic.\end{theorem}

\begin{definition}  Let $S$ be a smooth immersed self transverse surface in the smooth 4-manifold $Z$.  A \emph{finger move} is the operation of regularly homotoping a disc in $S$ along an embedded arc to create a pair of new transverse self intersections.  A \emph{Whitney move} is a regular homotopy supported in a neighborhood of a framed Whitney disc to eliminate a pair of oppositely signed self intersections.  By an \emph{isotopy} of $S$ we mean a regular homotopy through self transverse surfaces.  In particular, no new self intersections are either created or cancelled.\end{definition}

See \cite{FQ} for more details of these operations.

\begin{proposition} (e.g. \cite{FQ} P. 19-20, \cite{Qu} P. 353)  Let $A$ and $B$ be smooth embedded surfaces in the smooth 4-manifold $Z$.  If $A$ is regularly homotopic to $B$, then up to isotopy, the regular homotopy can be expressed as the composition of finitely many \emph{finger moves}, \emph{Whitney moves} and \emph{isotopies}.  \qed\end{proposition}

\begin{remark} It is well known by  the usual reordering argument that   if $A$ is regularly homotopic to $B$, then the regular homotopy can be chosen to consist of finger moves followed by Whitney moves in addition to intermediate isotopies.  See \S4 \cite{Mi} for reordering in the context of Morse functions.\end{remark}










\section{Shadowing regular homotopies by tubed surfaces}\label{shadow}

In this section we show that if $f_0:A_0\to M$ is an embedding of a smooth surface with embedded transverse sphere $G$ into a smooth 4-manifold  $M$ and  $f_t:A_0\to M$ is a generic \emph{regular homotopy} supported away from $G$, then $f_t$ can be \emph{shadowed} by a \emph{tubed surface}.  Roughly speaking there is a smooth \emph{isotopy} $g_t:A_0\to M$ with $g_0(A_0)=f_0(A_0)$ such  that when $f_t$ is self transverse,  $g_t(A_0)$ is approximately $f_t(A_0) $ with tubes connecting to copies of $G$.  As $f_t(A_0)$ undergoes a finger or Whitney move, $g_t(A_0)$ changes by isotopy that adds tubes that connect to copies of $G$ or modifies existing tubes.  In particular, if $f_1(A_0)=A_1$ is an embedding, then $A_0$ is isotopic to $A_1$ with tubes connecting to copies of $G$. Sections \S6 - \S8 are about how to eliminate or normalize these tubes by isotopy.  In this paper there are different types of tubes.  Tubes may follow arcs in a surface as in the proof of Lemma \ref{light bulb trick}, but they may follows paths in $M$ away from $A_1$.

This section is motivated by the following lemma.  

\begin{lemma} \label{finger shadow} Let $R$ be a connected embedded smooth surface in the smooth 4-manifold $M$.  If $R$ has an embedded transverse sphere $G$ and $R_1$ is obtained from $R$ by a finger move, then $R$ is isotopic to a surface $R_2$ consisting of $R_1$ tubed to two parallel copies of $G$.  \end{lemma}

\begin{proof}  Let $z=R\cap G$, $x, y \in R\setminus G$ and $\kappa$ a path from $y$ to $x$ with $\inte(\kappa)\cap (R\cup G)=\emptyset$.   Let  $\sigma\subset R\setminus y$ be an embedded path from $x$ to $z$ and $R_t$ a regular homotopy starting at $R$ corresponding to a finger move along $ \kappa$ that is supported very close to $\kappa$.  Then $R_1$ has two points of self intersection and a Whitney move along the obvious Whitney disc $E$ undoes the finger move, up to isotopy supported in the neighborhood of the finger.  

Let $D$ be a small  disc transverse to $R$ with $D\cap R=x$.  Let $T$ be the disc disjoint from $R$ which is the union of a tube that starts at $\partial D$ and follows $\sigma$ and then attaches to a parallel copy $G'$ of $G\setminus\inte(N(z))$ disjoint from $G$.  The tube should lie very close to $\sigma $ and $G'$ should be very close to $G$.  Let $R_2$ be the embedded surface obtained from $R_1$ by removing two discs $D', D''$ and attaching parallel copies of $T$  as in Figure 5.1.  The key observation is that $R_2$ is isotopic to $R$ via an isotopy supported near the union of $T$ and the 4-ball $B$ which is the support of the finger move.  This isotopy is essentially the following one.  If $F=I\times I \times 0\cup I\times I \times \epsilon\cup 0\times I\times [0,\epsilon]$ and $H$ the closed complement of $F$ in 
$\partial (I\times I\times [0,\epsilon])$, then up to rounding corners, $F$ is isotopic to $ H$ via an isotopy supported in $I\times I\times [0,\epsilon]$.  In our setting the product between the two copies of $T$ corresponds to $I\times I\times [0,\epsilon]$ and, after a small isotopy, $N(E)\cap R_2$  corresponds to $F$.  Now isotope $F$ to the corresponding $H$.  The resulting surface is readily seen to be isotopic to $R$ via an isotopy supported in $B$.\end{proof}

\begin{figure}[ht]
\includegraphics[scale=0.60]{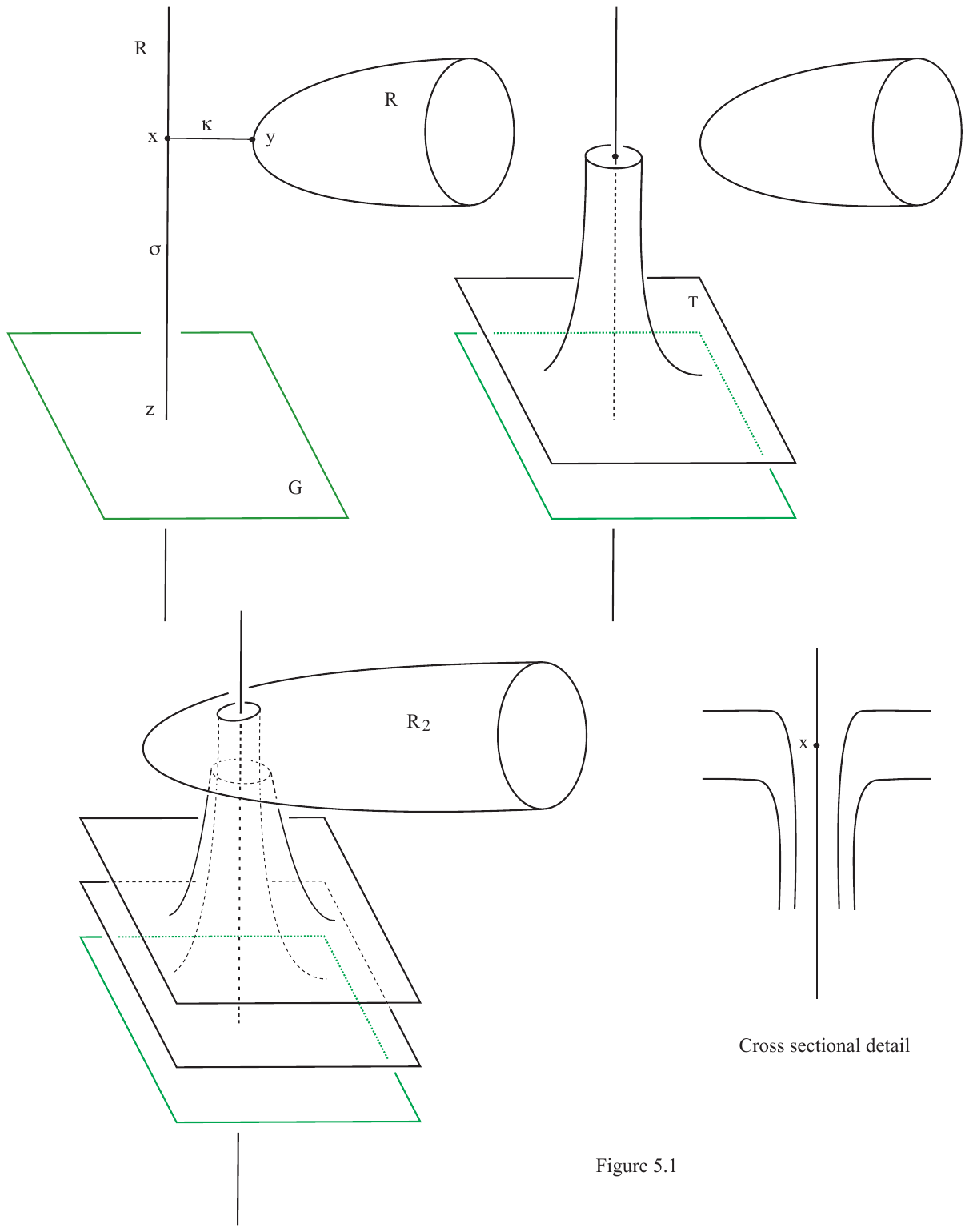}
\end{figure}

\begin{remarks}  Reversing the above isotopy gives an isotopy  $R'_t$  from $R$ to $R_2$.  Both $R_t$ and $R'_t$ start at $R$, the former ends at an immersed surface and the latter at an embedded one $R'_1$  obtained from the immersed one by removing some discs near the double points and tubing off with copies of $G$.    In what follows we will have  a regular homotopy $R_t$ that starts at an embedded surface $R$.  We will construct an isotopy  $R'_t$ starting at $R$ such that, except at times near finger and Whitney moves, the immersed surface will be closely approximated by an embedded surface $R'_t$ not counting a multitude of discs, which are essentially tubes that connect to parallel copies of $G$.  For reasons of organization, we define a \emph{tubed surface} as the data needed to define such a surface with tubes, rather than the surface itself.  Unlike the model case of Lemma \ref{finger shadow}, the tubes may follow intersecting paths in the surface or embedded paths away from the surface.  Note that if a tube $T_1$ (resp. $T_2$) follows path $\kappa_1\subset R'_t$ (resp. $\kappa_2\subset R'_t$) and $\kappa_1\cap \kappa_2\neq\emptyset$, then $T_1\cap T_2 =\emptyset$, provided that one of the tubes is  closer to $R'_t$ than the other.  

If $R_1$ is embedded, then the resulting tubed surface $R'_1$, (or more precisely what we call its \emph{realization}) will look like $R_1$ together with a jumble of tubes connecting to parallel copies of $G$.  Most of this paper is about how to clean up the mess, i.e. to show that under appropriate hypotheses this surface is actually isotopic to $R_1$ or some normal form.  The rest of the chapter is organized as follows.  We first give a definition of \emph{tubed surface}, second describe how to construct an embedded surface, called the \emph{realization}, from tubed surface data, third describe moves on the data that give isotopic surfaces, fourth use all this to show how to shadow a regular homotopy and fifth describe tubed surfaces that are in \emph{normal form}.  \end{remarks}

\begin{remarks} \label{two vs four}Given an embedded  path $\phi\subset M$ with $\phi\cap R=\partial \phi$ and  whose ends are orthogonal to the possibly disconnected surface $R$, then we can tube $R\setminus \inte(N(\partial \phi))$ by attaching an annulus $T$ that follows $\phi$.  Up to isotopy supported in $N(\phi)$ there are two ways to do this if $R$ is oriented and the new surface maintains the orientation induced from $R$.  That is because $\pi_1(SO(3))=\BZ_2$ and we can insist by construction that for $t\in I, T\cap B^3\times t$ is an equator where $N(\phi)=B^3\times I$.  While the resulting two surfaces are equivalent as unparametrized surfaces, we keep the distinction since we may want to tube up other things such as links.  In general there are four ways up to isotopy.  Thus we have the next definition which  enables us to keep track of how to attach tubes to pairs of circles and more generally to attach pairs of tubes to pairs of Hopf bands.  \end{remarks}

\begin{definition} \label{framed arc} A \emph{framed embedded path} is a smooth embedded path $\tau(t), t\in [0,1]$ in the 4-manifold $M$ with a framing $\mF(t)=(v_1(t), v_2(t), v_3(t))$  of its normal bundle.  Let $(C(0), x(0))$ consist of a smooth embedded circle $C(0)$, with base point $x(0)$, lying in the normal disc to $\tau$ through $\tau(0)$ that is spanned by the vectors $(v_1(0), v_2(0))$  with $x(0)$  lying in direction $v_1(0)$.  Define $(C(t), x(t)$) a smoothly varying family having similar properties for each  $t\in [0,1]$.   Call the annulus $(C(t), x(t))$, $t\in [0,1]$ the \emph{cylinder connecting} $C(0)$ and $C(1)$.  It should be thought of as lying very close to  $\tau$.  \end{definition}


As a warmup to the following long definition, the reader is encouraged to look at Figures 5.8 and 5.9 which show how different types of tubes can arise in the course of an isotopy and thus the need for an elaborate definition. See Figures 5.2 - 5.5 which show some of the data in the definition of a \emph{tubed surface} and exhibit realizations of this data.

\begin{definition}\label{tube surface} A \emph{tubed surface} $\mA$ in the 4-manifold $M$ consists of

i) a generic self transverse immersion $f:A_0\to M$, where $A_0$ is a closed surface based at $z_0$ with $A_1$ denoting  $f(A_0)$.  The preimages $(x_1, y_1), \cdots, (x_n, y_n)$ of the double points are pairwise ordered.  $A_0$ is called the \emph{underlying} surface and $A_1$ the \emph{associated} surface to $\mA$.

ii)  An embedded transverse 2-sphere $G$ to $A_1$, with $A_1\cap G=z=f(z_0)$.  

iii)  For each $i=1, \cdots, n$, an immersed path $\sigma_i\subset A_0$ from $x_i$ to $z_0$. See Figure 5.2 for views in $A_0$ and $A_1$.

iv)  immersed paths $\alpha_1, \cdots, \alpha_r$ in $A_0$ with both endpoints at $z_0$ and for each $i=1, \cdots, r$,  pairs of points $(p_i, q_i)$ with $p_i\in \alpha_i$ and $q_i\in A_0$ and a framed embedded path, $\tau_i\subset M$ from $f(p_i)$ to $f(q_i)$ with $\inte(\tau_i)\cap (G\cup A_1)=\emptyset$.  See Figure 5.4.

v) pairs of immersed paths $(\beta_1, \gamma_1), \cdots, (\beta_s, \gamma_s)$ in $A_0$ where $\beta_i$ goes from $z_0$ to $b_i $ and $\gamma_i$ goes from $g_i $ to $z_0$ and framed embedded paths $\lambda_i\subset M$ from $f(b_i)$ to $f(g_i)$ with $\inte(\lambda_i)\cap (G\cup A_1)=\emptyset$.  See Figure 5.5.

Curves of the form $\sigma_i, \alpha_j, \beta_k, \gamma_l$ are called \emph{tube guide} curves and the $\tau_p$ and $\lambda_q$ curves are called \emph{framed tube guide} curves.  The union of all of these curves is  the \emph{tube guide locus}.  The $\sigma_i, \alpha_j, \beta_k, \gamma_l$ curves are required to be self transverse and transverse to each other with interiors disjoint from the $z_0, x_i, q_j, b_k, g_l$ points and disjoint from the $p_j$ points except where required in iv).  At points of intersection and self intersection of these curves, except $z_0$, one curve is determined to be \emph{above} or \emph{below} the other curve.  The various points $z_0, x_i, y_j, p_k, q_l, b_m, g_n$ are all distinct.

The curves $\tau_i$ and $\lambda_j$ are pairwise disjoint, disjoint from $G$ and intersect $A_1$ only at their endpoints.  Additional  conditions on the framings of the $\tau_i$ and $\lambda_j$ curves will be given in Construction \ref{tube construction}.  They dictate the placement of the $C(0)$ and $C(1)$ curves.  This ends the definition of a tubed surface.\end{definition}

\begin{remark}  The data i) - iii) are what's needed to create a tubed surface arising from a finger move as in Lemma \ref{finger shadow}.  Data iv) and v) are needed to describe tubed surfaces arising from Whitney moves.  Crossings of tube guide curves may occur in preparation for Whitney moves and in the process of transforming pairs of double tubes to pairs of single tubes in \S 6. \end{remark}
\vskip 8pt
We now show how a tubed surface gives rise to an embedded surface. 

\begin{construction}  \label{tube construction} Associated to the tubed surface $\mA$ construct an embedded surface $A$, called the \emph{realization} of $\mA$  as follows.  For each $i$, remove from $A_1$ the image of a small $D^2$ neighborhood of $y_i$.  Attach to $f(\partial D^2)$ a disc $D(\sigma_i)$ consisting of a tube $T(\sigma_i)$ that follows $f(\sigma_i)$ and connects to a slightly pushed off copy of $G\setminus \inte(N(z))$.  See Figure 5.2. These copies of $G$ are sufficiently close to $G$ so that the closed product region between each of them and $G$ is disjoint from all the framed tube guide curves.  If $u\in \sigma_i\cap \sigma_j$, $u\neq z_0$, and $\sigma_i$ lies above $\sigma_j$ at $u$, then near $f(u)$, construct $T(\sigma_j)$ to lie closer to $A_1$ than does $T(\sigma_i)$.  With abuse of notation, this allows for the case $i=j$.  See Figure 5.3.  Let $\hat A$ be the embedded surface thus far constructed.  
\begin{figure}[ht]
\includegraphics[scale=0.60]{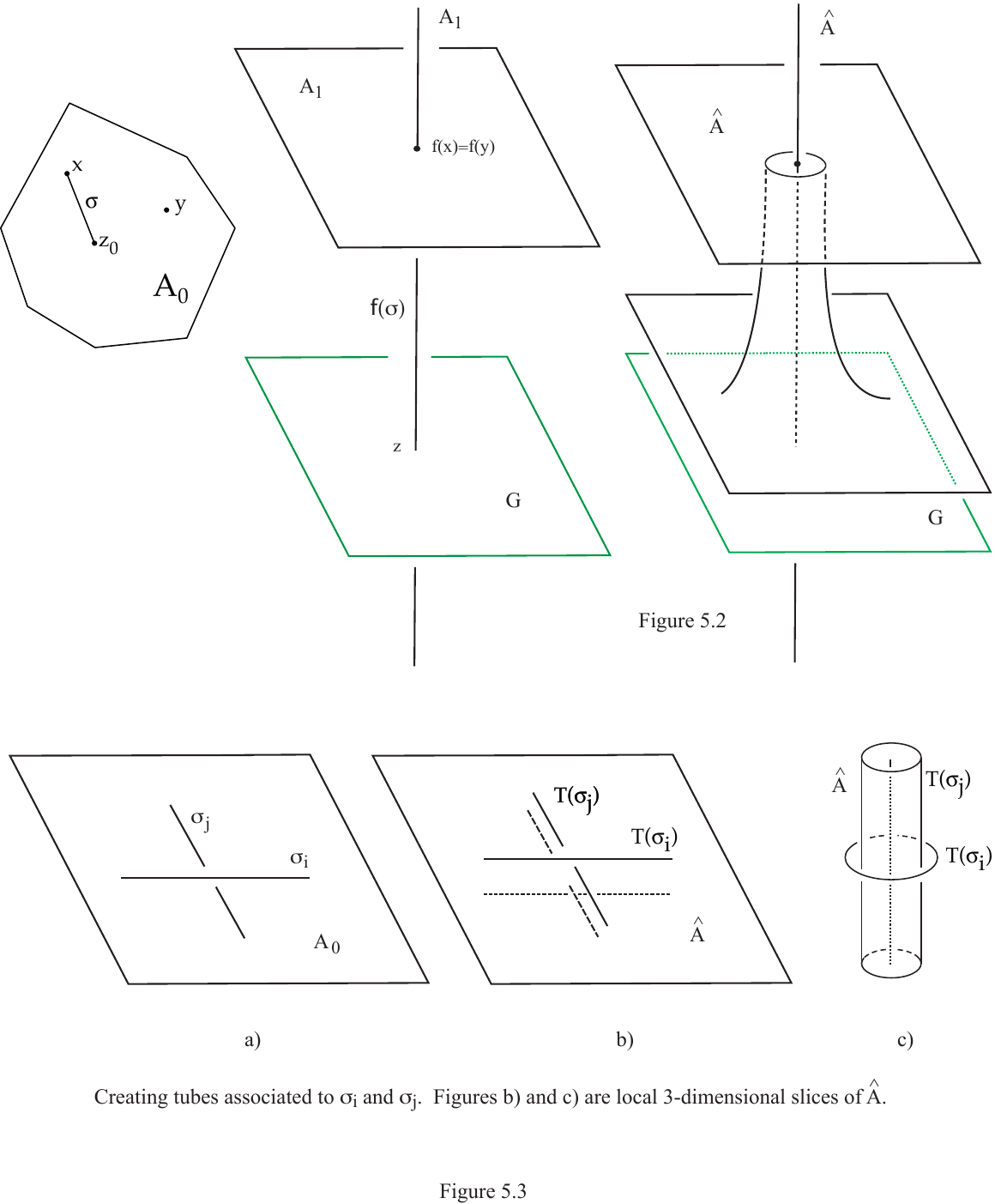}
\end{figure}

In a similar manner, associated to the path $\alpha_i$ is a 2-sphere $P(\alpha_i)$ with $P(\alpha_i)\cap A_1=\emptyset$,  consisting of two parallel copies  of $G\setminus \inte(N(z))$ connected by a tube $T(\alpha_i)$ that follows the path $f(\alpha_i)$. Again these copies of $G$ are sufficiently close to $G$ that the closed product region between each of them and $G$ is disjoint from all the framed tube guide curves.   Next attach a tube $T(\tau_i)$ following the framed embedded path $\tau_i$ from $C(0)=P(\alpha_i)\cap\partial N(\tau_i)$ to $C(1)=\hat A\cap \partial N(f(q_i))$.  Note that the previous condition implies that $T(\tau_i)$ attaches to the tube part of $P(\alpha_i)$.  Here we assume that $\tau_i$ approaches $f(p_i)$ normally to $\hat A$ and is parametrized by $[-1/4,1]$ and framed so that restricting to $[0,1]$,\ $C(0)$ (resp. $C(1)$) is in the plane spanned by $v_1(0)$ and $v_2(0)$ (resp. $v_1(1)$ and $v_2(1)$) as in Definition \ref{framed arc}.  This assumption is the condition on the framing on $\tau_i$ that is required but not explicitly stated at the end of Definition \ref{tube surface}. The tube $T(\tau_i)$ is called a \emph{single tube}.  See Figure 5.4.  Let $\hat A'$ the embedded surface constructed at this stage.
\begin{figure}[ht]
\includegraphics[scale=0.60]{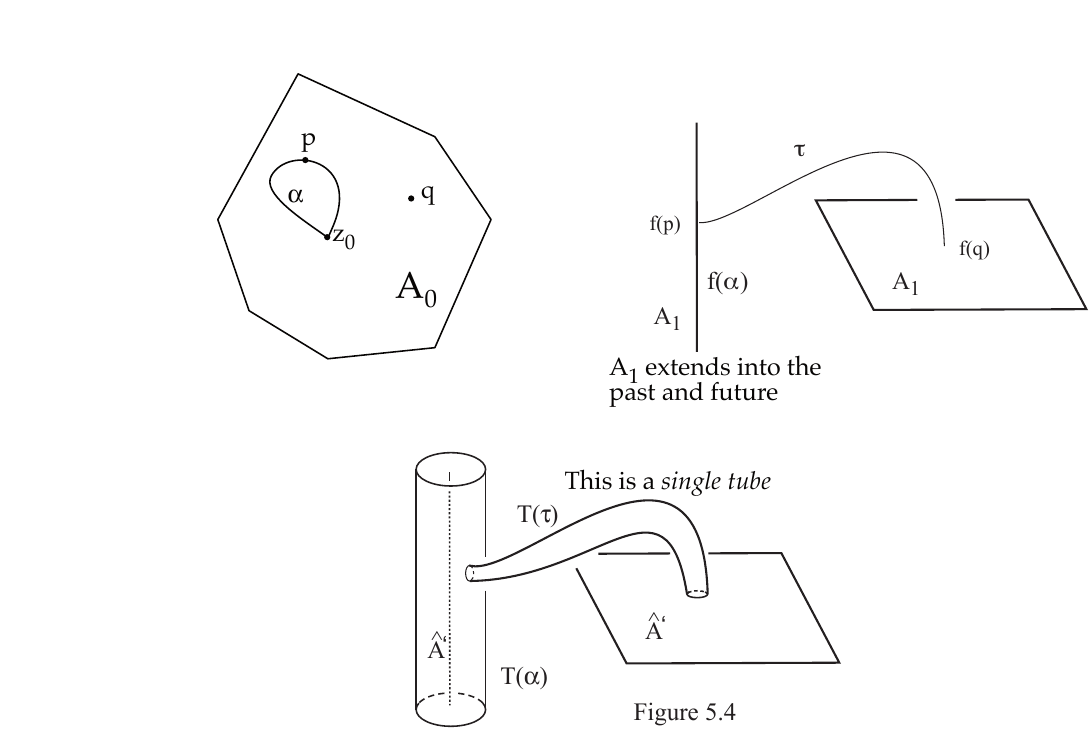}
\end{figure}

Next for each $i$, construct discs $D(\beta_i)$ and $D(\gamma_i)$ consisting of slightly pushed off copies of $G\setminus \inte(N(z))$ tubed very close to and respectively along $f(\beta_i)$ and $f(\gamma_i)$ with boundary  lying in discs normal to $\hat A'$ at $f(b_i) $ and $f(g_i)$.   Roughly speaking the rest of the construction of $A$ from $\hat A'$ proceeds as follows.  Appropriately sized 4-balls $N(f(b_i))$ and $N(f(g_i))$ have the property that $\partial N(f(b_i)))\cap (\hat A'\cup D(\beta_i))$ and $\partial N(f(g_i))\cap (\hat A'\cup D(\gamma_i))$ are Hopf links.  Connect these links by  tubes that parallel $\lambda_i$ such that $\partial N(f(g_i))\cap \hat A'$ (resp. $\partial N(f(b_i))\cap \hat A'$) connects to $\partial N(f(b_i))\cap D(\beta_i)$  (resp. $\partial N(f(g_i))\cap D(\gamma_i)$).  See Figure 5.5.  
\begin{figure}[ht]
\includegraphics[scale=0.60]{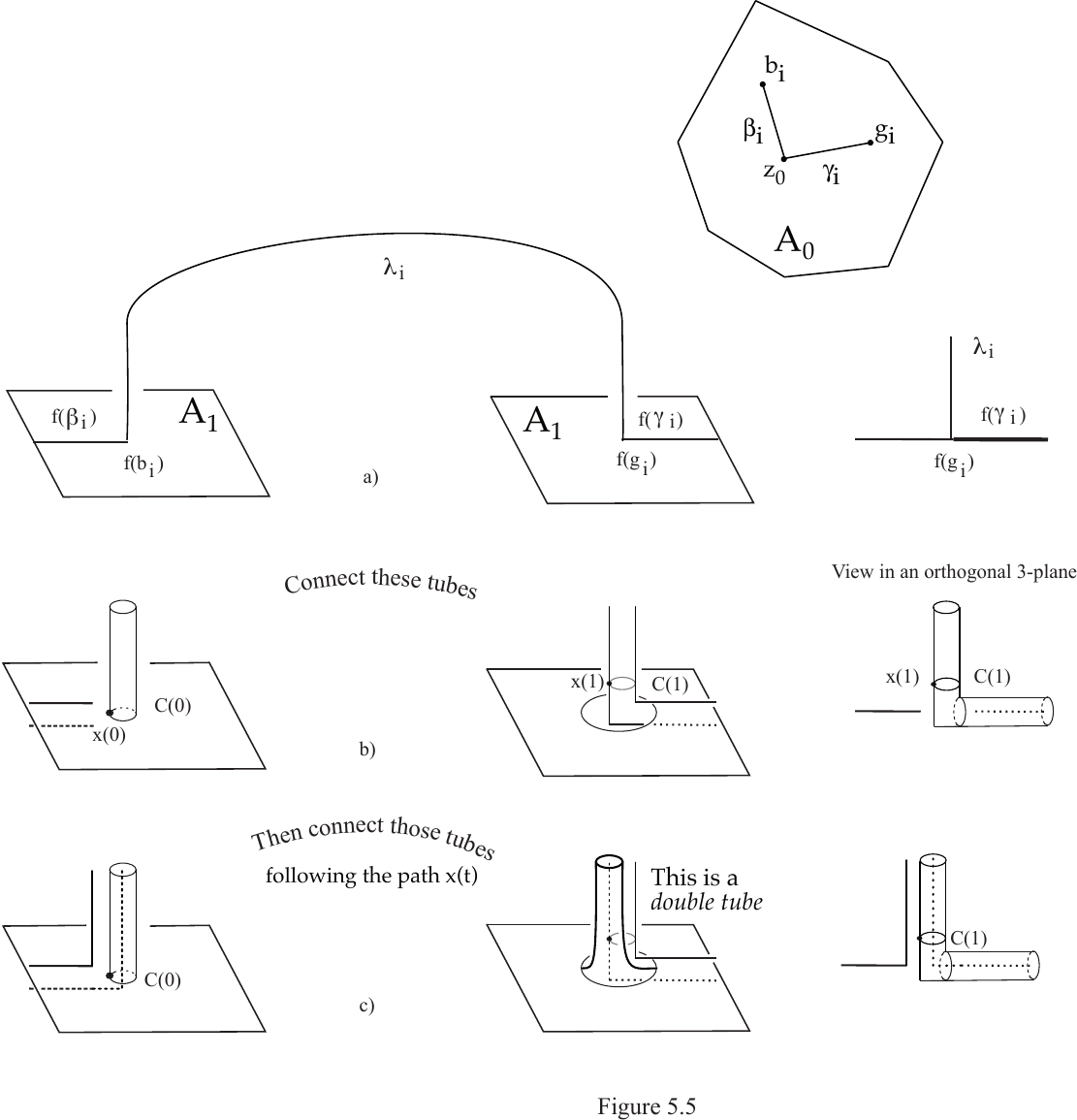}
\end{figure}

More precisely, delete $\inte(N(f(b_i)))$ from $D(\beta_i)$ and continue to call $D(\beta_i)$ the disc that remains.  Next remove $\inte(1/2N(f(b_i)))$ from $\hat A'$ and let $C(0)=\partial((1/2N(f(b_i)))\cap \hat A')$.  Also remove $\inte(N(f(g_i)))$ from $\hat A' $ and $\inte(1/2N(f(g_i)))$ from $D(\gamma_i)$ and call $D(\gamma_i)$ what remains.  Now bend $D(\gamma_i)$ near $\partial D(\gamma_i)$ in the direction of $\lambda_i$ and then let  $C(1)=\partial D(\gamma_i)$.  See Figure 5.5 b).  Suppressing the epsilonics, $1/2N(f(g_i))$ and $1/2N(f(b_i))$ are \emph{half radius} 4-balls about $f(b_i)$ and $f(g_i)$ and with respect to that scale, the tube of $D(\beta_i)$ is very close to $f(\beta_i)$ and the tube of $D(\gamma_i)$ is very close to $f(\gamma_i)$ together with a short segment of $\lambda_i$.

We assume that $\lambda_i$ approaches $\hat A'$ in geodesic arcs near $f(g_i)$ and $f(b_i)$, and in the two 3-planes spanned by these arcs and $\hat A'$, it approaches $\hat A'$ orthogonally. 
We assume that $\lambda_i$ is parametrized by $[0,1]$ and framed so that $C(0)$ is in the plane spanned by $v_1(0)$ and $v_2(0)$ as in Definition \ref{framed arc}.  Also $x(0)=C(0)\cap \beta_i$ and $v_1(0)$ points towards $x(0)$.  Again $C(1)=\partial D(\gamma_i)$ with $x(1)$ the point indicated in Figure 5.5 b) and assume that  $C(1)$ lies in the plane spanned by $v_1(1)$ and $v_2(1)$ with $x(1)$ lying in the arc spanned by $v_1(1)$.  

As in Definition \ref{framed arc} use $\lambda_i$ to build a tube connecting $C(0)$ and $C(1)$.  Using a tube that follows the path $x(t), t\in [0,1]$ connect $\partial D(\beta_i)$ to $\partial N(g_i)\cap \hat A'$  as in Figure 5.5 c).  The union of these two tubes, a Hopf link $\times I$, is called a \emph{double tube}.  This completes the construction of the realization $A$ of $\mA$. \end{construction}

\begin{remark} \label{tubes unknotted} The single and double tubes do not \emph{link} with other parts of the realization.  In particular,  except for the spots where they attach to and/or near the associated surface $A_1$, the single tubes and double tubes stay a uniformly bounded distance away from $A_1$ and the transverse sphere $G$.  Further, the tubes following the $\sigma, \alpha, \beta, \gamma$ curves stay  within this distance to the associated surface and the parallel copies of $G$ also stay  within this uniform distance to $G$.\end{remark}

We now describe operations on a tubed surface $\mA$ that correspond to isotopies of the realizations.

\begin{definition} \label{tube sliding} We enumerate  \emph{tube sliding moves} on a tubed surface $\mA$ corresponding to redefining the location and crossing information of tube guide curves in the underlying surface $A_0$.

i)  Type 2), 3) Reidemeister moves on tube guide curves.  See Figure 5.6 a).

ii) Reordering tube guide curves near $z_0$.  See Figure 5.6 b).  


iii)  Sliding a tube guide curve across a double point.  See Figure 5.6 c).  There are two cases depending on whether or not the tube guide $\kappa$ lies in the sheet through $y_i$ or the sheet through $x_i$.  In the former case we require that $\kappa\neq\sigma_i$.  

iv) Sliding across a tube guide curve $\kappa$ across a double tube.  See Figure 5.6 d).  Here $\kappa\neq \gamma_i$ (resp. $\beta_i)$ can slide over $b_i$ (resp. $g_i$). 

v)  Sliding tube guide curves across a single tube.  Here a tube guide curve $\kappa\neq \alpha_i$ can slide across $q_i$ and over $p_i$.  Any tube guide curve can slide under $p_i$.  See Figure 5.6 e).\end{definition}

\begin{figure}[ht]
\includegraphics[scale=0.60]{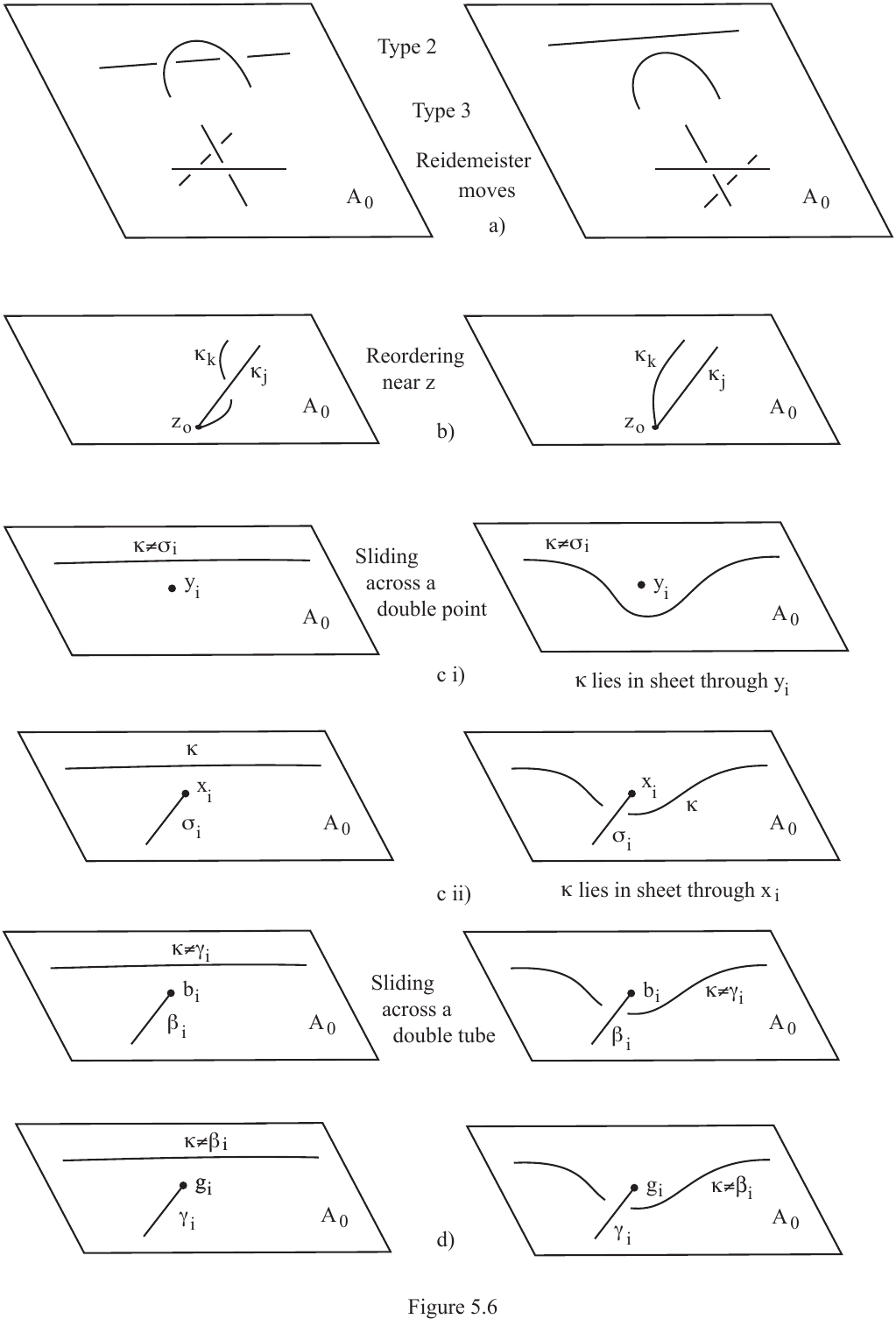}
\end{figure}

\begin{figure}[ht]
\includegraphics[scale=0.60]{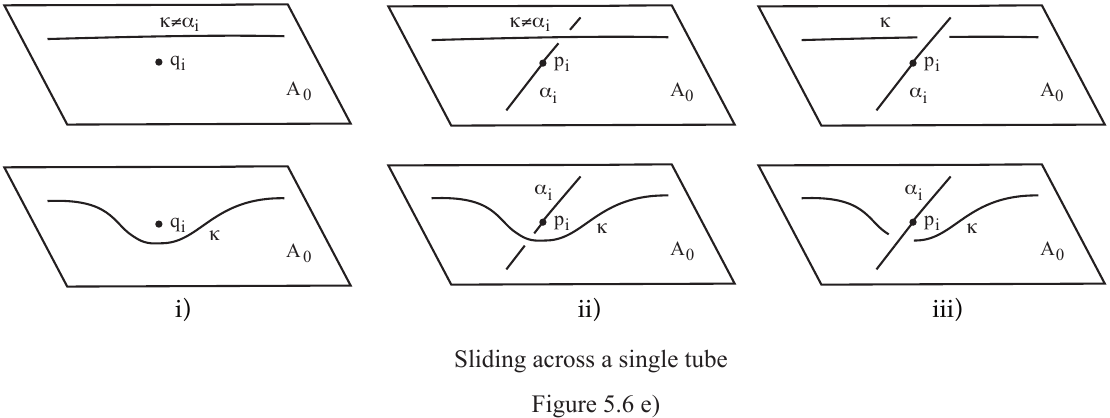}
\end{figure}

\begin{remark}  Sliding $\sigma_i$ across $y_i$ has the self referential problem analogous to a handle sliding over itself.  Similarly for sliding $\beta_i$ (resp. $\gamma_i$, resp. $\alpha_i$) across $g_i$ (resp. $b_i$, resp. $q_i$).\end{remark}  

\begin{lemma}\label{tube sliding lemma}  If  $\mA$ and $\mA'$ are tubed surfaces that differ by tube sliding, then their realizations $A$ and $A'$ are isotopic. \end{lemma}

\begin{proof}  We consider the effect on the realization of $\mA$ by the various tube sliding moves.  The under/over crossing data in $A_0$ reflects how close one tube is to $A_1$ compared with the another.  As the Reidemeister 2), 3) moves respect this closeness it follows that they induce an isotopy from $A$ to $A'$.    

Next we consider reordering near $z_0$.  Since $G$ has a trivial normal bundle, there is an $S^1$ worth of directions that it can push off itself.  These directions correspond to the directions that the image of tube guide curve $f(\kappa)\subset A_1$ can approach $z$.   We can assume that the various parallel copies of $G\setminus\inte(N(z))$ are equidistant from $G$ at angle that of the angle of approach of the various $f(\kappa)$'s.   Let $D\subset A_1$ denote a disc which is a small neighborhood of the bigon that defines the reordering.   Let $K_i$ denote the disc consisting of a parallel copy $G_i$ of $G\setminus \inte(N(z))$ together with its tube that follows the arc $f(\kappa_i)\cap D$.  If $\kappa_j$ is above $\kappa_k$ as in Figure 5.6 b) and $K'_j$ and $K'_k$ are the discs resulting from the reordering, then there is an isotopy of $A$ to $A'$ supported on  $K_k$  where  $G_k$ is first pushed radially close to $G$, then  rotated to the angle defined by $f(\kappa'_k)$ and then pushed out.   Here $\kappa'_k$ is the reordered $\kappa_k$.

The proof that $A$ is isotopic to $A'$ for the operations of Figures 5.6 c i), d), e i) are all the same.  Here we are sliding a tube in $A$ that parallels a curve $f(\kappa)\subset A$ across a disc.     In the case of Figure 5.6 c i) (resp. Figure 5.6 d)) that disc includes the disc $D(\sigma_i)$ (resp.  $D(\gamma)$ or $D(\beta)$).  In the case of Figure e i) that disc includes $P(\alpha_i)$ minus an open disc.

The proof that $A$ is isotopic to $A'$ for the operations of Figures 5.6 c ii) and e iii) are the same and are local operations.  Here we are sliding a tube paralleling a curve $f(\kappa)\subset A$ across a small disc.  The slid tube is very close to $A$, closer than other tubes in the vicinity.

The proof that $A$ is isotopic to $A'$ for the operation of Figure 5.6 e ii)  requires Lemma \ref{light bulb trick} for we want Remark \ref{tubes unknotted} to continue to hold.  Let $T(\tau_i)$ (resp. $T(\kappa))$ denote the tube in $A$ that parallels $\tau_i$ (resp. $\kappa\neq \alpha_i$).  Sliding $T(\kappa)$ over $f(p)$ entangles $T(\tau_i)$ with $T(\kappa)$.  The entangled $T(\tau_i)$ is the tube being isotoped using Lemma \ref{light bulb trick}.  Lemma \ref{light bulb trick} requires that there be a path, in the notation of that lemma, from $y$ to $z$.   This requires that $\kappa\neq\alpha_i$. \end{proof}

We now define operations on tubed surfaces corresponding to finger and Whitney moves.

\begin{definition/construction}  Let $A_1$ be the associated surface to the tubed surface $\mA$.  To a generic finger move from $A_1$ to $A_1'$ with corresponding regular homotopy from $f$ to $f'$ we obtain a new tubed surface $\mA'$ said to be obtained from $\mA$ by a \emph{finger move}.   By \emph{generic} we mean that the support of the homotopy is away from all the framed tube guide curves and images of tube guide curves of $\mA$. $ \mA'$ will have the same underlying surface $A_0$ as $\mA$ and $A_1'$ will be its associated surface.  Let $(x_1,y_1)$ and $(x_2, y_2)$ be  the new pairs of $f'$ preimages of double points in $A_0$, where both $f'(x_1)$ and $f'(x_2)$ (resp. $f'(y_1)$ and $f'(y_2)$) lie in the same local sheet of $A_1'$.  Let $\sigma_1$ and $\sigma_2$ be parallel embedded paths from $x_1$ and $x_2$ to $z_0$ transverse to the existing tube guide paths.   The tube guide locus of $\mA'$ consists of that of $\mA$ together with  $\sigma_1$ and $\sigma_2$ where all crossings of these $\sigma_i$'s with pre-existing tube guide curves are under crossings.  See Figure 5.7.\end{definition/construction}
\begin{figure}[ht]
\includegraphics[scale=0.60]{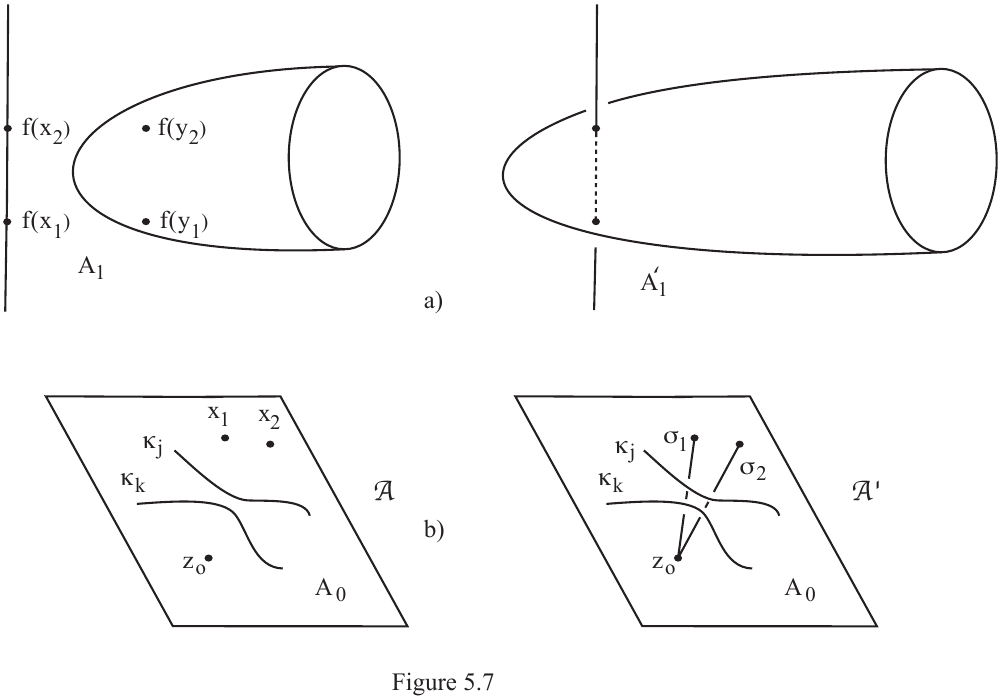}
\end{figure}

\begin{remark}  There is flexibility in the construction of $\mA'$ from $\mA$ in the choice of which pair of points are called $x_i$ points and in the choice of the $\sigma_i$ paths.\end{remark}

\begin{lemma}  \label{finger shadow one} If $\mA'$ is obtained from $\mA$ by a finger move, then their associated realizations are isotopic.\end{lemma}

\begin{proof}  This lemma is a restatement of Lemma \ref{finger shadow} in our setting.  The proof is the same after recognizing that the support of the isotopy in the target is contained in the support of the finger move  together with a small neighborhood of the product region between the discs $D(\sigma_1)$ and $D(\sigma_2)$, notation as in Construction \ref{tube construction}.  \end{proof}

\begin{definition/construction}  Let $A_1$ be the associated surface to the tubed surface $\mA$.  A Whitney move from $A_1$ to $A_1'$ corresponding to the regular homotopy from $f$ to $f'$ with Whitney disc $w$ is said to be \emph{tube locus free} if $\inte(w)$ is disjoint from the framed tube guide curves of $\mA$ and $\partial w$ intersects the images of tube guide curves of $\mA$ only at double points of $A_1$.  Let $(x_1,y_1), (x_2, y_2)$ denote the pairs of points in $A_0$ corresponding to these double points with notation consistent with that of Definition \ref{tube surface} and let $E_1, E_2$ denote the local discs involved in the Whitney move.  We say that the Whitney move is \emph{uncrossed} if both $f(x_1)$ and $f(x_2)$ lie in the same $E_i$ and \emph{crossed} otherwise.  If $w$ is an uncrossed Whitney disc, then we obtain the tubed surface $\mA'$ as indicated in Figure 5.8 and if $w$ is crossed, then $\mA'$ is obtained as in Figure 5.9.  Accordingly $\mA'$ is said to be obtained from $\mA$ by an \emph{uncrossed} or \emph{crossed} Whitney move.\end{definition/construction}

\begin{remark}  An uncrossed Whitney move gives rise to a single tube while a crossed Whitney move gives rise to a double tube.  In the former case two $\sigma$ curves become an $\alpha$ curve.  In the latter case, the $\sigma$ curves become $\beta$ and $\gamma$ curves. \end{remark}

\begin{lemma}  \label{whitney shadow} If $\mA'$ is obtained from $\mA$ by a tube locus free Whitney move, then its realization is isotopic to that of $\mA$.\qed\end{lemma}

\begin{figure}[ht]
\includegraphics[scale=0.70]{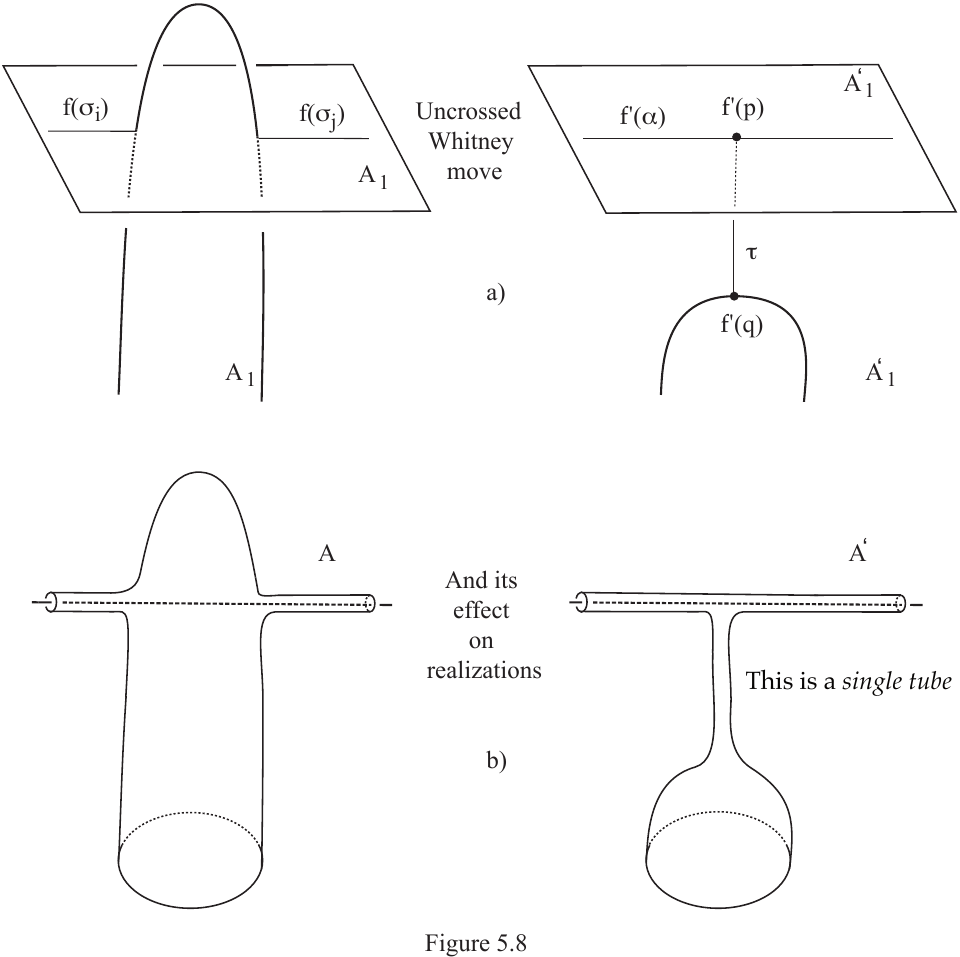}
\end{figure}
\begin{figure}[ht]
\includegraphics[scale=0.70]{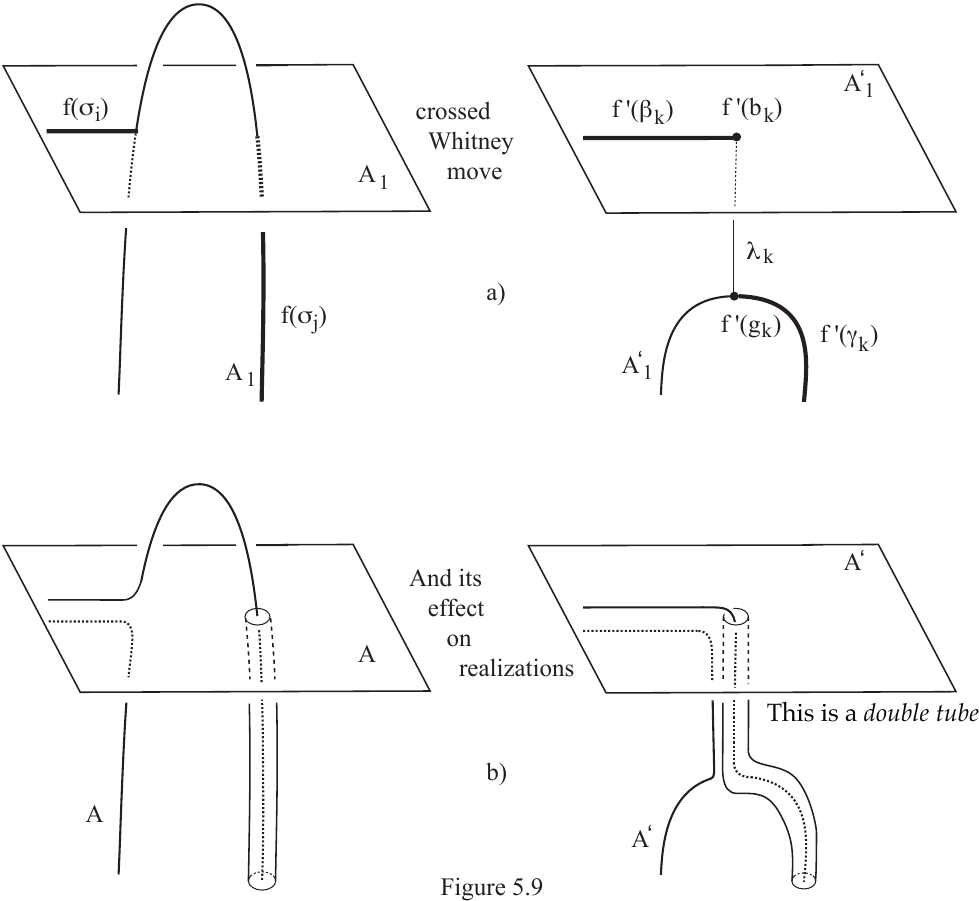}                    
\end{figure}

\begin{definition} \label{elementary isotopy}  We define an \emph{elementary tubed surface isotopy}, or \emph{elementary isotopy} for short, on the tubed surface $\mA$ as any of the following operations on $\mA$.

a)  The defining data changes smoothly without combinatorial change.  In particular, at no time are there new tangencies or new intersections among the various objects.

b) tube sliding moves.

c) finger moves

d) tube locus free Whitney moves\end{definition}


\begin{lemma} \label{elementary} If  $\mA$ and $\mA'$ are tubed surfaces that differ by an elementary  isotopy, then their realizations are isotopic. \end{lemma} 

\begin{proof} a) Smooth changes in the defining data induce smooth changes in the realizations.

b) This is Lemma \ref{tube sliding lemma}.

c) This is Lemma \ref{finger shadow}.

d) This is illustrated in Figures 5.8 and 5.9. \end{proof}

\begin{definition}  \label{shadow surface} Let $R_0$ be an immersed surface in the smooth 4-manifold $M$ with embedded transverse sphere $G$.  Let $f_t: R\to M^4$ be a generic regular homotopy supported away from $G$ which is a self-transverse immersion except at times $\{t_i\}$ where $0=t_0<t_1< \cdots <t_{m-1}<t_m=1$ and $i\notin\{0,m\}$.   Let $R_0=f_0(R), R_1, \cdots, R_m=f_1(R)$ be such that  for $i=1, \cdots, m-1$, $R_i$ is a surface $f_s(R)$ for some $s\in (t_{i}, t_{i+1})$.  We say that the regular homotopy $f_t$ is \emph{shadowed by tubed surfaces} if there exists a sequence $\mR_0, \mR_1, \cdots, \mR_m$ of tubed surfaces such that for all  $i, R_i$ is the associated surface to $\mR_i$ and for $i\neq m$, $\mR_{i+1}$ is obtained from $\mR_i$ by   elementary isotopies.  The tubed surfaces $\mR_0, \mR_1, \cdots, \mR_m$ are called  $f_t$-\emph{shadow tubed surfaces}.   \end{definition}

\begin{theorem} \label{shadow isotopy} If $f_t:A_0\to M$ is a generic regular homotopy with $f_0(A_0)$ an embedded surface, $M$ a smooth 4-manifold, $G$ a transverse embedded sphere to $f_0(A_0)$ and $f_t$ is supported away from $G$, then $f_t$ is shadowed by tubed surfaces.\end{theorem}

\begin{proof} Define the tubed  surface $\mA_0$ as follows.  $A_0$ is the underlying surface, $f_0:A_0\to M$ is the self transverse immersion with $f_0(A_0)=A^0_1$ the associated surface. The tube guide locus $=\emptyset$.  Let $0<t_1< \cdots <t_{m-1}<1$ be the non self transverse times of $f_t$.  If there are no singular times in $[s,s']$ and a tubed surface $\mA_s$ has been constructed whose associated surface $A_1^s=f_s(A_0)$, then an elementary isotopy of type a) transforms it to one with $A_1^{s'}=f_{s'}(A_0)$.  Thus we need only show how to shadow regular homotopies near non self transverse  times.   Now each non self transverse  time corresponds to either a finger or Whitney move.  Since $f_0(A_0)$ is embedded,  $t_1$ is the time of a finger move.  The shadowing of the initial finger move is given in the proof of Lemma \ref{finger shadow one}.  More generally, that lemma shows how to shadow any  finger move.  By induction we  assume that the conclusion holds through time $t$, where $t_{k-1}<t<t_{k}$ and $t_k$ is the time of a Whitney move.

Let $\mA_{k-1}$  denote the tubed surface with underlying surface $A_1^{k-1}=f_t(A_0)$.  By Lemma \ref{elementary} we can assume that $t$ is a time just before the Whitney move.  Let $w$ be a Whitney disc for the Whitney move.  Being 2-dimensional we can assume that $w$ is disjoint from the framed tube guide curves to $\mA_{k-1}$.  Suppose that $w$ cancels the $f_t$ images of the  points $\{u,v,u',v'\}\subset A_0$, where  $f_t(u)=f_t(v)$ and $f_t(u')=f_t(v')$.  Here $u$ and $u'$ (resp. $v$ and $v'$) are the endpoints of disjoint arcs $\phi$ and $\phi'$ in $A_0$ which map to $\partial w$ under $f_t$ and $(u,v) = (x_i,y_i)$ and $(u',v')=(x_j,y_j)$, notation as in Definition \ref{tube surface}.  By switching $\phi$ and $\phi'$ and/or $i$ and $j$ if necessary, we can assume that the first equality is that of an ordered pair and the second is setwise.

Next we show that after tube sliding moves $w$ becomes a tube locus free Whitney disc, i.e. no tube guide curve crosses $\inte\phi$ or $\inte\phi'$.  Now $x_i\in \partial \phi$, so all the tube guide crossings with $\inte(\phi)$ can be eliminated by sliding across the double point $x_i$.  If $x_j\in \partial \phi'$, then we can clear $\inte(\phi')$ of tube guide curves in a similar manner.  If $\partial\phi'=(y_i, y_j)$ and the tube guide $\kappa$ crosses $\inte(\phi')$ we clear it from $\phi'$ as follows.  Since $i\neq j$, it follows that $ \kappa\neq \sigma_p$ for some $p\in \{i,j\}$. Apply a sequence of Reidemeister 2) moves supported in a small neighborhood of $\phi'$ to make $\kappa$  adjacent to $y_p$ and  then slide it across the double point $y_p$.  

Since $w$ is now a tube locus free Whitney disc we can shadow the Whitney move by an uncrossed (if $u'=x_j$) or crossed (if $u'=y_j$) Whitney move.  Thus $\mA_{k}$ is obtained from $\mA_{k-1}$ by a sequence of tube sliding moves and a tube locus free Whitney move.\end{proof}

\begin{remarks}   i)  If the regular homotopy $f_t$ is of the form finger moves followed by Whitney moves, then  the  tubed surface following the finger moves can be chosen so that the curves $\sigma_i$ are embedded and pairwise disjoint away from $z_0$.

ii)  If $f_1(A_0)$ is embedded, then the final tubed surface has no $\sigma_i$ curves. 

iii)  There is no restriction on the surface $A_0$.  Below and in the next section we require that $A_0$ be a 2-sphere.\end{remarks}



The rest of this section is relevant for  4-manifolds $M$ whose fundamental group has 2-torsion.  See Figure 5.10.

\begin{figure}[ht]
\includegraphics[scale=0.70]{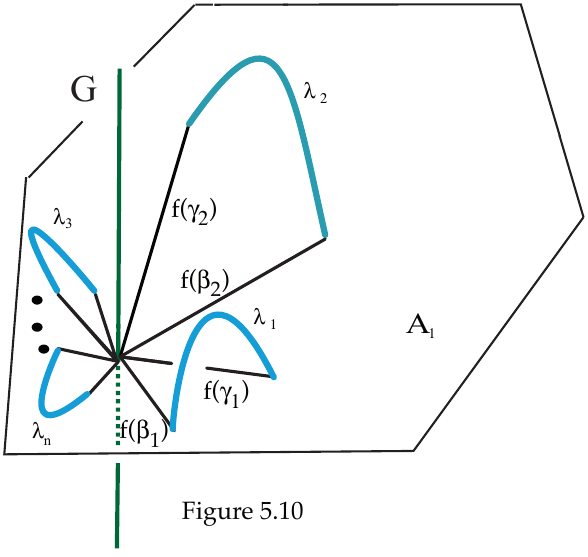}
\end{figure}

\begin{definition} \label{normal form} We say that the tubed surface $\mA$ is in \emph{normal form} if in addition to the conditions of Definition \ref{tube surface} we have

a) The immersion $f:A_0 \to M$ is an embedding with associated surface $A_1=f(A_0)$ and $A_0$ is a 2-sphere.

b) There are no $\alpha$ curves.

c) The paths $\beta_1, \gamma_1, \beta_2, \gamma_2, \cdots, \beta_n, \gamma_n$ are embedded and  cyclicly arrayed in $A_0$ about the common endpoint $z_0$.

d) The framed embedded paths $\lambda_1, \lambda_2, \cdots, \lambda_n$ represent distinct nontrivial 2-torsion elements of $\pi_1(M)$.   

We say that the surface $A$ is in \emph{normal form} with respect to the embedded surface $A_1$ \emph{representing} elements $[\lambda_1], \cdots, [\lambda_n]$ if $A$ is the realization of $\mA$ with data as in  a)-d).  We say two normal forms are \emph{equivalent} if they represent the same set of elements.   \end{definition}

The following result is immediate by our realization construction.

\begin{proposition}\label{all realized} If $A_1$ is an embedded 2-sphere with transverse sphere $G$ and  given any finite set of distinct nontrivial 2-torsion elements of $\pi_1(M)$, then there exists an embedded 2-sphere $A$ in normal form with respect to $A_1$ representing these elements of $\pi_1(M)$.\qed\end{proposition}

\begin{lemma}  \label{normal permutation} The isotopy class of the realization is independent of the cyclic ordering of the $(\beta_i, \gamma_i)$'s arrayed about $z_0\in A_0$.\end{lemma}

\begin{proof} Using the tube sliding operations any two adjacent pairs $(\beta_i, \gamma_i), (\beta_{i+1},\gamma_{i+1}) $  can be permuted.\end{proof} 

\begin{remark} In \S8 we will show that the isotopy class is also independent of the choice of framings on the framed tube guide curves.  Thus equivalent normal form surfaces are isotopic, as stated in Theorem \ref{2-torsion}. \end{remark}

\section {From double tubes to single tubes}  \label{from}

In what follows $\mA$ is a tubed surface in $M$ with realization $A$ whose associated surface $A_1$ is an embedded 2-sphere.    In this section we show that $\mA$ can be transformed to $\mA'$ with isotopic realizations without changing $A_1$ so that, if $\pi_1(M)$ has no 2-torsion, then $\mA'$ has at most  one  double tube, which is homotopically inessential.  If $\pi_1(M)$ has 2-torsion, then $\mA' $ may additionally have finitely many double tubes representing distinct elements of order two.  This is done by appropriately replacing pairs of double tubes with  pairs of single tubes. 

\begin{lemma}  \label{basepoint}Suppose that $R_0$ and $R_1$ are spheres with common transverse sphere $G$ in the 4-manifold M.  If $R_0$ and $R_1$ coincide near $G$ and are homotopic in $M$, then they are homotopic via a homotopy whose support in $M$ is disjoint from $G$, in particular the homotopy is basepoint preserving.\end{lemma}

\begin{proof} Let $K:S^2\times I\to M$ be a homotopy from $R_0$ to $R_1$.  Let $E\subset R_0 $ be a disc containing $R_0\cap G$ that  coincides with a disc of $R_1$.  After an initial homotopy  making $K$ transverse to $G$ we see that $K^{-1} (G)$ is 1-manifold with  exactly one component $\psi$ that goes from $S^2\times 0$ to $S^2\times 1$.  By the 3-dimensional light bulb theorem $\psi$ is unknotted, so  after reparametrization of $S^2\times I$, we can assume that $K$ is a basepoint preserving homotopy.   Next, replace $K$ by another, also called $K$, such that there exists a disc $D\subset S^2$ where $K_t|D$ is independent of $t$ with $K_0(D)=\frac{1}2 E$.

To complete the proof it suffices to show that if a sphere $R_3$ in $M\setminus G$ is homotopically trivial in $M$ it is homotopically trivial in $M\setminus G$.  The relevant $R_3$ is obtained from $R_0\cup R_1$ by deleting $\frac{1}2 E$ from each and gluing along the resulting boundaries.  Let $\tilde M$ denote the universal covering of $M$ and $\tilde G$ the preimage of $G$.  By Lemma \ref{pi injectivity} the universal cover of $M\setminus G$ is $\tilde M\setminus \tilde G$, hence if $\tilde R_3$ denotes a lift of $R_3$ to $\tilde M$, it suffices to show that $\tilde R_3$ is homologously trivial in $\tilde M\setminus \tilde G$.  Let $Z\subset \tilde M$ be a 3-chain transverse to $\tilde G$ with $\partial Z=\tilde R_3$.  Since $H_1(\tilde G)=0, [Z\cap \partial N(\tilde G)]=[(Z\cap \tilde G)\times \partial D^2]=0\in   H_2(\partial N(\tilde G))$ and hence by replacing  $Z| N(\tilde G)$ by one supported in $\partial N(\tilde G)$ with the same boundary, there exists a 3-chain $Z'$ with $\partial Z'=\tilde R_3$ and $Z'\cap \tilde G=\emptyset$.  \end{proof}

\begin{remark} If $\mA$ is a tubed surface whose associated surface $A_1$ is an embedded sphere, then we can view double tubes as representing elements of $\pi_1(M)$ as follows.  First $A_1$ itself can be viewed as the basepoint for $ \pi_1(M)$.  Recall Lemma \ref{pi injectivity}. Next a double tube corresponds to a path $ \lambda_i$ in $M$ from some $b_i\in A_1$ to some $g_i\in A_1$.  Thus the double tube gives rise to an element of $\pi_1(M)$.\end{remark}

\begin{definition} \label{tube orientation} Let $\mA$ be a tubed surface with realization $A$.  Let $\kappa$ denote one of $\sigma_i, \beta_j$ or $\gamma_k$ and $y$ the corresponding $x_i, b_j$ or $g_k$.  If we compress the tube in $A$ that follows $f(\kappa)$ near $f(y)$ we obtain an immersed surface,  one component of which is an embedded 2-sphere $P=P(\kappa)$, that is homotopic to the transverse sphere $G$.  This sphere has an induced orientation that coincides with $A$ away from the compressing disc.  Define $\epsilon(P)=\epsilon(\kappa)=+1$ if $[P]=[G]$ and $-1$ otherwise. Here $M, A$ and $G$ are oriented so that  $<A, G>=+1$. 

Similarly compressing $A$ near a point of $\tau_i$ gives rise to an embedded surface one component of which is an embedded 2-sphere $P(\alpha_i)$ isotopic to two oppositely oriented copies of $G$ tubed together along $\alpha_i$.   \end{definition} 


\begin{lemma} \label{signs} If $P=P(\kappa)$ is constructed as above and $D$ is the compressing disc that splits off $P$ and oriented to coincide with that of $P$, then $\epsilon(P)=<D',A>$  Here $D'$ is $D$ shrunk slightly to have boundary disjoint from $A$.  \qed\end{lemma}

\begin{lemma}  If $\mA$ is a tubed surface, then for all $i$, \ $[P(\beta_i)]=-[P(\gamma_i)]=\pm [G]\in H_2(M)$ and for all $j$, \   $[P(\alpha_j)]=0$.  \qed\end{lemma}

\begin{remark}  Up to isotopy there are four framings on the framed embedded path $\lambda_i$, hence four ways of constructing a double tube from $\lambda_i$.  See Figure 5.5 and Remark \ref{two vs four}.  Note that two give $1=\epsilon(P(\beta_i))=-\epsilon(P(\gamma_i))$ while two give $-1=\epsilon(P(\beta_i))=-\epsilon(P(\gamma_i))$.\end{remark}


\noindent\textbf{Sign Convention:}  By switching $\beta_i$ and $\gamma_i$, if necessary, we can assume that $\epsilon(\beta_i)=-1$ and $\epsilon(\gamma_i)=+1$.  Orient $\beta_i$ to point from $z$ to $b_i$, $\lambda_i$ to point from $f(b_i)$ to $f(g_i)$ and $\gamma_i$ to point from $g_i$ to $z$.

\vskip 8pt

We next calculate $[A]\in \pi_2(M)$.  Since $A$ and $G$ are 2-spheres, each distinct element of $\pi_1(M)$ gives rise to a distinct pair of geometrically dual spheres in $\tilde M$, the universal cover of $M$, that projects to the pair $(A,G)$.  Thus, $\pi_2(M)=\pi_2(\tilde M)=H_2(\tilde M)$ which contains the group ring  $\mH= H_2(G)\pi_1(M)$ as a submodule.  
Since $\mA$ is a tubed surface, $[A]$ lies in the coset $\mH+[A_1]$.    Since $\pi_1(A_1)=0$, each $\lambda_i$ determines a well defined element, $[\lambda_i]\in \pi_1(M)$.  With the above conventions the triple $(\beta_i, \lambda_i,\gamma_i)$ gives rise to the element $ [G][\lambda_i]-[G][\lambda_i]^{-1}\in \mH$.  On the other hand, each $\alpha_j, \tau_j$ gives rise to the trivial element.  We therefore have:

\begin{lemma} \label{homology formula} $[A]=[A_1]+\sum_{i=1}^s [G][\lambda_i]-[G][\lambda_i]^{-1}\in \pi_2(M)$.\qed\end{lemma}

\begin{remarks} \label{pairing} i)  If $A$ is homotopic to $A_1$, then $ \sum_{i=1}^s [G][\lambda_i]-[G][\lambda_i]^{-1}=0$.  Therefore, if $[\lambda_i]=1$ whenever $[\lambda_i]^2=1$ holds, then we can reorder the $\lambda_i$'s so that $[\lambda_1]=[\lambda_2]^{-1}, \cdots, [\lambda_{2p-1}]=[\lambda_{2p}]^{-1}$ and $[\lambda_s]=1$ if $s=2p+1$.

ii)  In the 2-torsion case we can reorder the $\lambda_i$'s so that $[\lambda_1]=[\lambda_2]^{-1}, \cdots, [\lambda_{2p-1}]=[\lambda_{2p}]^{-1}$,  and $[\lambda_{2p+1}], [\lambda_{2p+2}], \cdots, [\lambda_s]$ represent distinct 2-torsion elements of $\pi_1(M)$, though possibly $[\lambda_s]=1$.  Distinctness follows by observing that a 2-torsion element is its own inverse, hence all but at most one of the $\lambda_i$'s representing the same 2-torsion element lies within the first $2p$ elements.   \end{remarks}
\vskip 8pt

The following is the main result of this section.   It's the crucial point in the proof of Theorem \ref{main} where the no 2-torsion condition is used.  It also essentially uses that $A_1$ is a 2-sphere.  

\begin{proposition} \label{double to single} Let $\mA$ be a tubed surface in the 4-manifold $M$ whose associated surface $A_1$ is an embedded sphere homotopic to the realization $A$ of $\mA$.  Let $G$ denote the transverse sphere to $A_1$.  Then via an isotopy supported away from $G$, $A$ is isotopic to the realization $A'$ of a tubed surface $\mA'$ with associated surface $A_1$ such that if $\pi_1(M)$ has no 2-torsion, then $\mA'$ has at most one double tube, in which case the unique double tube is homotopically trivial.  If $\pi_1(M)$ has 2-torsion, then $\mA'$ additionally may   have $n\ge 0$  double tubes, each representing distinct  2-torsion elements of $\pi_1(M)$.
\end{proposition}

\begin{proof} By Remark \ref{pairing} we can reorder the $\lambda_i$'s so that $[\lambda_1]=[\lambda_2]^{-1}, \cdots, [\lambda_{2p-1}]=[\lambda_{2p}]^{-1}$ and $[\lambda_{2p+1}], [\lambda_{2p+2}], \cdots, [\lambda_s]$ are distinct 2-torsion elements of $\pi_1(M)$ with possibly $[\lambda_s]=1$.  Thus if $\pi_1(M)$ has no 2-torsion, then $s=2p$ or $2p+1$ in which case all but at most one of the double tubes are paired up.  The remaining one, if it exists,  is homotopically trivial.  We will show that an isotopy of $A$ transforms $\mA$ to $\mA'$ with the same $A_1$ but with the double tubes $\lambda_1, \lambda_2$ eliminated.    The proposition then follows by induction on the number of double tubes.   

We now  decipher Figure 6.1, Figure 6.2 b), Figure 6.3 and 6.4 which are crucial to the proof. Each figure describes the intersection of a properly embedded planar surface with a 4-ball  $B^3\times I$.  In all cases the surface consists of two components one called $F$ which lies entirely in the present $B^3 \times \frac{1} 2$ and the other $E$.  In Figures 6.1 a),b) and 6.3 a),b) $F$ is an annulus which looks like a square with a tube attached, while in Figures 6.2 b) and  6.4 a),b),c)  $F$ looks like a square with a pair of pants attached, hence is a pair of pants.  In  Figures 6.1 a),b) and Figures 6.3 a),b) $E$ is an annulus comprised of a square $E_1$ with a tube $U$ attached.  Here $E_1$ is of the form $\alpha \times I$, where $\alpha\times \frac{1} 2\subset B^3 \times \frac{1} 2$ is a properly embedded horizontal arc.  If $x=E_1\cap F$, then $U$  follows an arc $\beta\subset F$ from $x$ to some $y\in F\cap B^3 \times \frac{1} 2$.  Thus $E$ is obtained by  attaching $U$ to $E_1\setminus \inte(D(x))$ along $\partial D(x)$, where $D(x)$ is a neighborhood of $x$ in $E_1$.  Note that $U$ is defined similarly to part  of $T(\sigma_i)$ defined in Construction \ref{tube construction}.  In both Figure 6.1 and Figure 6.3 the dark, possibly dotted, lines indicate $E\cap B^3 \times \frac{1} 2$.  The shaded regions indicate the parts of $U$ which lie in the future or past.  Note that $E$ is symmetric with respect to reflection in the $I$-factor.  The $E$ of Figure 6.2 b) is similar, except that it is a pair of pants consisting of two tubes attached to $E_1$.

We now describe the pairs of pants $E$ of Figures 6.4 a),b),c), henceforth respectively denoted $E_a, E_b, E_c$.  Let $\delta_a,\delta_b, \delta_c$  denote properly embedded  horizontal arcs in the $F$'s of Figures 6.4 a),b),c), each parallel to the bottom edge of the square referred to in the previous paragraph.   Let $U_a, U_b, U_c$ denote properly embedded tubes in $B^3\times I$, following the arcs $\delta_a, \delta_b, \delta_c$.   Up to a small isotopy of $E_a$ near the junction of $U_a$ and $t_a$, our $E_a, E_b, E_c$ are obtained by respectively attaching  tubes $T_a, T_b, T_c$ between $U_a, U_b, U_c$ and $E_1$.  These tubes follow arcs $t_a, t_b, t_c$ respectively from $U_a, U_b, U_c$ to $E_1$.    Here  $t_a, t_b \subset B^3 \times \frac{1} 2$.   We again follow the convention that dark, possibly dotted, lines indicate the intersection of $E$ with $B^3\times I$.  Here the shaded regions indicate the parts of the joined $U_a\cup T_a, U_b\cup T_b, U_c\cup T_c$ which lie either in the future or past.  Note that both $E_a$ and $E_b$ are symmetric with respect to reflection in the $I$ factor.    $E_c$ is similarly constructed.  Here the projection of $t_c$ to $B^3 \times \frac{1} 2$ equals $t_b$.  It coincides with $t_b$ near $U_b$ and thereafter lies in the future.  Similarly the projection of $T_c$ to $B^3 \times \frac{1} 2$ equals $T_b$.  It coincides with $T_b$ near $U_b$ and away from a small neighborhood of the coincidence locus lies completely in the future.

We resume the proof of the Proposition.  To start with we consider another model for a  double tube as shown in Figure 6.1.  See Remarks \ref{figure 6.1} i) for further explanation of this figure. 


\begin{figure}[ht]
\includegraphics[scale=0.60]{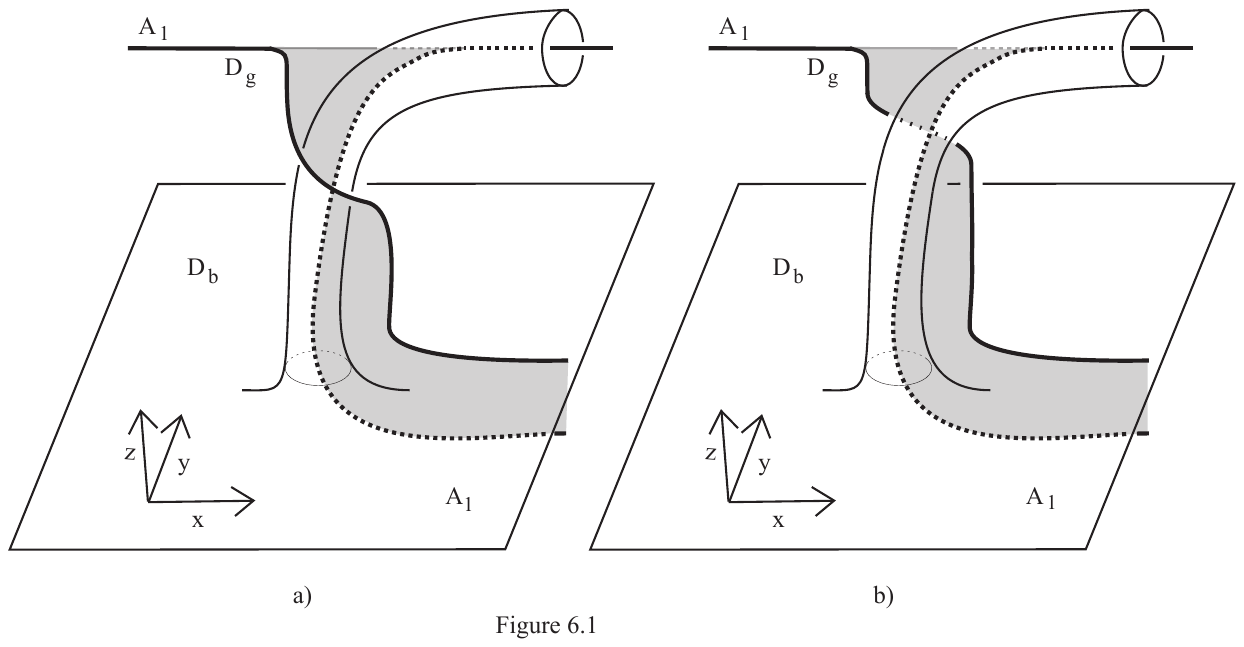}
\end{figure}

\begin{remarks}  \label{figure 6.1} i)  Fixing an orientation on $M$ and $G$ our sign conventions determine an orientation on $A_1$ and hence $A\cap A_1$ as well as the orientations on $A$ near the ends of a double tube associated to $\lambda_i$.  The orientation near $g_i$  is induced from $A_1$ and the orientation on $T(\gamma_i)$ is determined from the fact that $\epsilon(\gamma_i)=1$.   Similarly, for the $b_i$ end of the double tube, except that $\epsilon(\beta_i)=-1$.  Therefore, by Remark \ref{two vs four},  up to isotopy, there are two rather than four, ways of constructing a double tube associated to  $\lambda_i$.  Representatives are shown in Figures 6.1 a) and b).  

ii) Figures 6.1 a) and b) each show the projection into the $x,y,z$ plane of a neighborhood of  a double tube associated to  $\lambda$.  This consists of tubes emanating from two discs $D_b$ and $D_g$ lying in $A_1$ that are respectively neighborhoods of $f(b)$ and $f(g)$, where $\lambda$ connects $f(b)$ to $f(g)$.    In the figure, $D_b$ lies in the $x,y$ plane and $D_g$ lies in the $x,t$ plane.  Here $D_g$ corresponds to the horizontal lines at the top of the subfigures.  The shaded regions are projections of the tube from $D_g$ into the $x,y,z$ plane.  The preimage of the interior of each shaded region consists of two components, one in the future and one in the past.  Except where it is twisted, this tube lies in the $x,z,t$ plane.  Its intersection with the $x,y,z$ plane is the union of the  thick solid and dashed lines.  

\end{remarks} 





Since $A_1$ has the transverse 2-sphere $G$ it follows from Lemma \ref{pi injectivity} that the induced map $\pi_1(M\setminus (A_1\cup G))\to \pi_1(M)$ is an isomorphism.  Since homotopy implies isotopy we can isotope $\lambda_1$ and $\lambda_2$ to be anti-parallel.  I.e. there exists an embedded square  $D$ with opposite edges respectively on $A_1$ and $\lambda_1$ and $\lambda_2$.  See Figure 6.2 a).  Here $E_1$ and $F_1$ denote the components of $A_1\cap N(D)$.  Note that $F_1$ lies in the present and $E_1$ is of the form $\alpha\times I\subset B^3\times I$ where $\alpha\times \frac{1} 2$ is a horizontal arc in the present.

We now show that it suffices to assume that $A\cap N(D)$ appears as in Figure 6.2 b).  Figure 6.3 shows how to isotope the surface to effect a change of the framed embedded path corresponding to the non trivial element of $\pi_1(SO(3))$.    Thus we can assume that $A$  appears near $\lambda_1$ and $\lambda_2$  as depicted in Figure 6.2 b).    Now $A\cap N(D)$ may fail to appear as in Figure 6.2 b) because the images of tube guide paths may cross the interior of $D\cap A_1$, however by doing tube sliding moves in a manner similar to those in the proof of Theorem \ref{shadow isotopy} we can clear such curves from the neighborhood.

  
\begin{figure}[ht]
\includegraphics[scale=0.60]{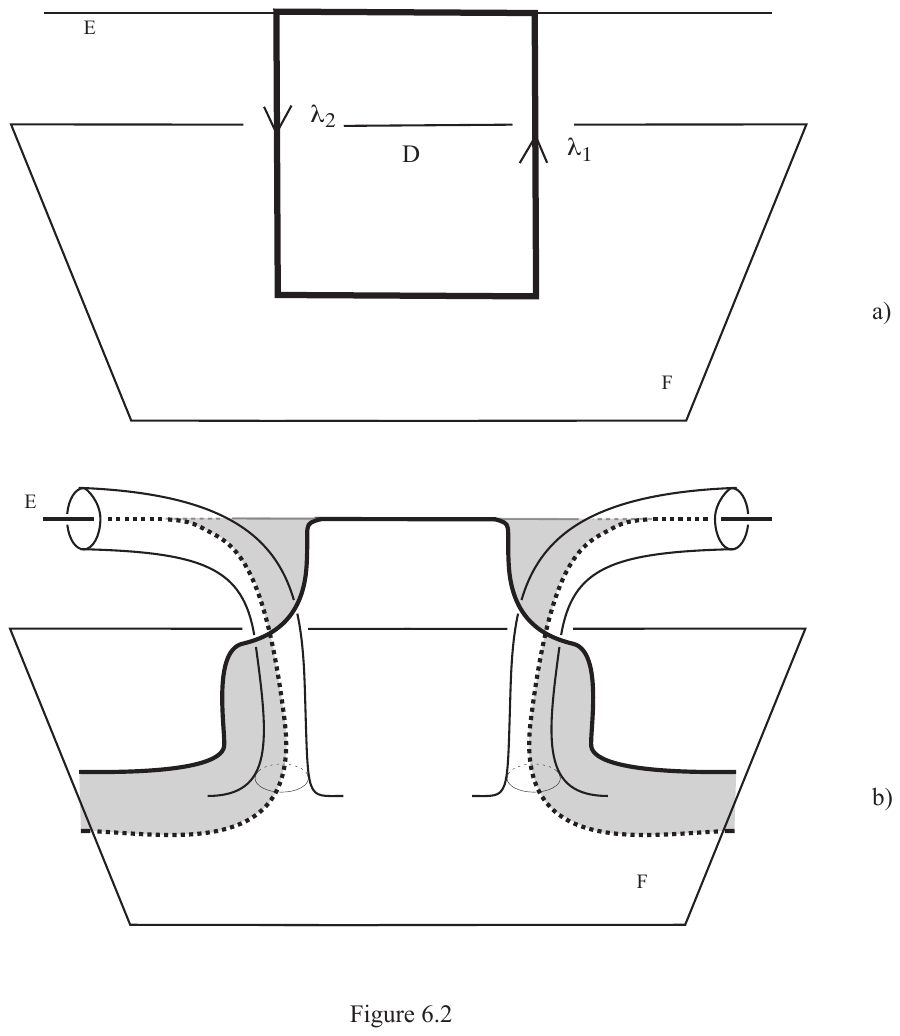}
\end{figure}
 
\begin{figure}[h]
\includegraphics[scale=0.60]{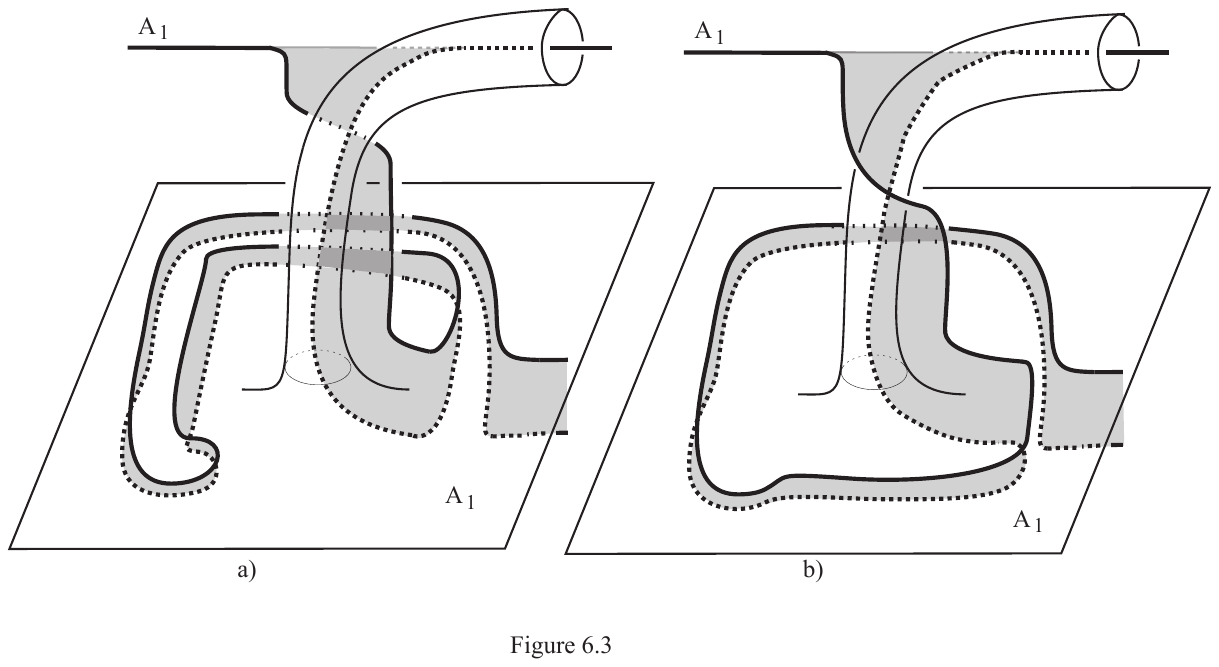}
\end{figure}

\begin{figure}[h]
\includegraphics[scale=0.60]{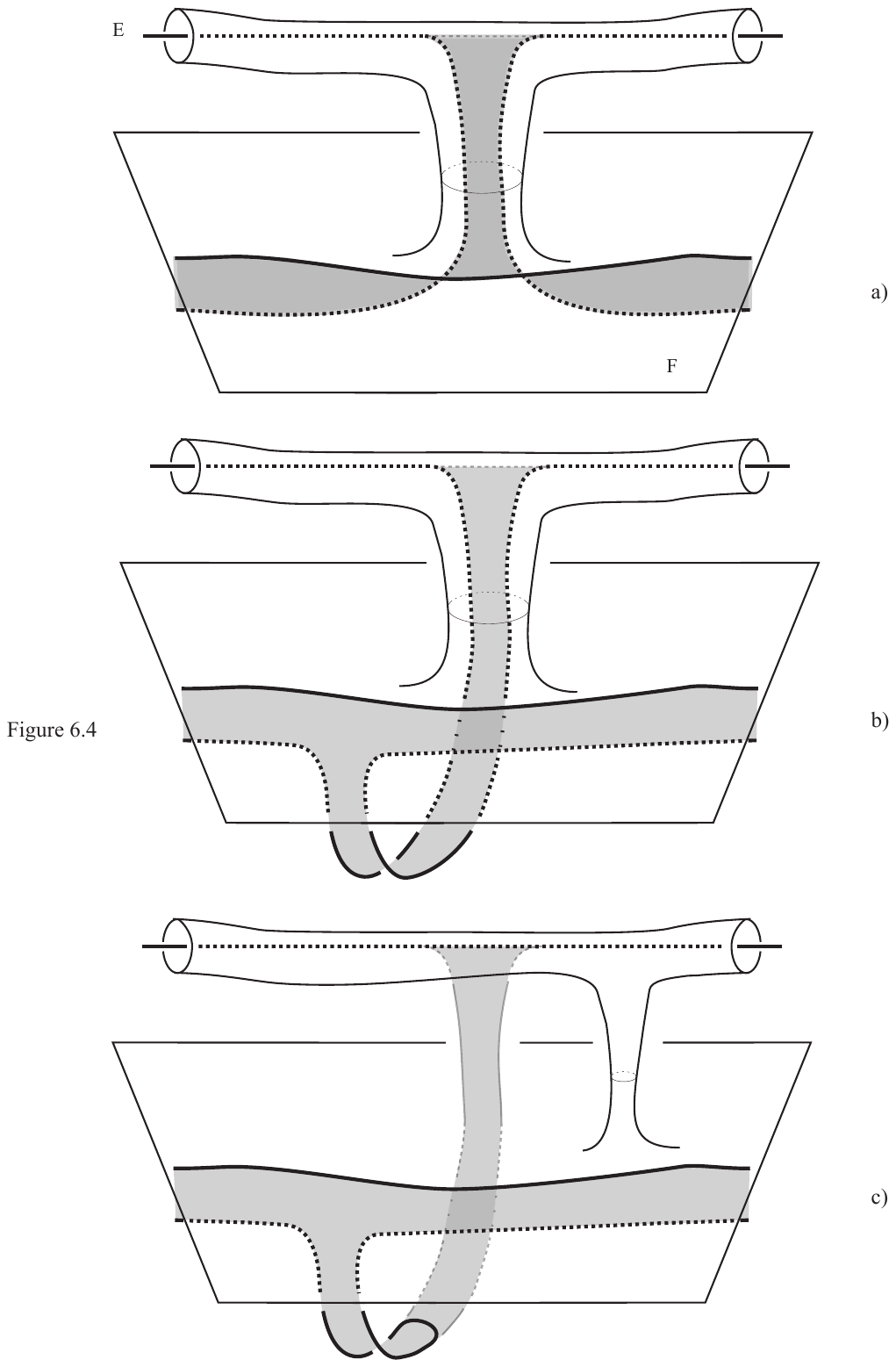}
\end{figure}

We now do an isotopy of $A$, supported in $N(D)$, that transforms a pair of double tubes as in Figure 6.2 b) into a pair of single tubes as shown in Figure 6.4 c).  Call the \emph{E-component} of $A\cap N(D)$ the one that intersects $E_1$ and call the other the \emph{F-component}.  Keeping in mind that in Figure 6.2 and Figure 6.4 the intersection of the $E$-component with the present is drawn in thick possibly dashed lines, isotope the $A$ of Figure 6.2 b) to that of Figure 6.4 a) via a $t$-coordinate preserving isotopy.  Note that during the isotopy the projections to the present of the $E$ and $F$ components can intersect provided the intersections avoid the thick lines.  Next isotope $A$ as in Figure 6.4 b) to that of Figure 6.4 c).  Note that this isotopy is supported in the $E$ component.   The isotopy from Figure 6.4 b) to Figure 6.4 c) is as follows.    Without changing the projection to the present, first push most of the tube we earlier called $T_b$ into the future and then isotope the $F$ component as indicated. \end{proof}

\section {Crossing Changes}  

In this section we show that crossing changes involving distinct tube guide curves do not change the isotopy class of the realization of a tubed surface.

\begin{lemma}(Crossing Change Lemma)\label{crossing} If the tubed surface $\mA'$ is obtained from the tubed surface $\mA$ by a crossing change involving either distinct tube guide paths or distinct components of $\alpha_i\setminus p_i$, then the corresponding realizations $A'$ and $A$ are isotopic via an isotopy supported away from the transverse surface $G$.  See Figure 7.1 a). \end{lemma}

\begin{figure}[ht]
\includegraphics[scale=0.60]{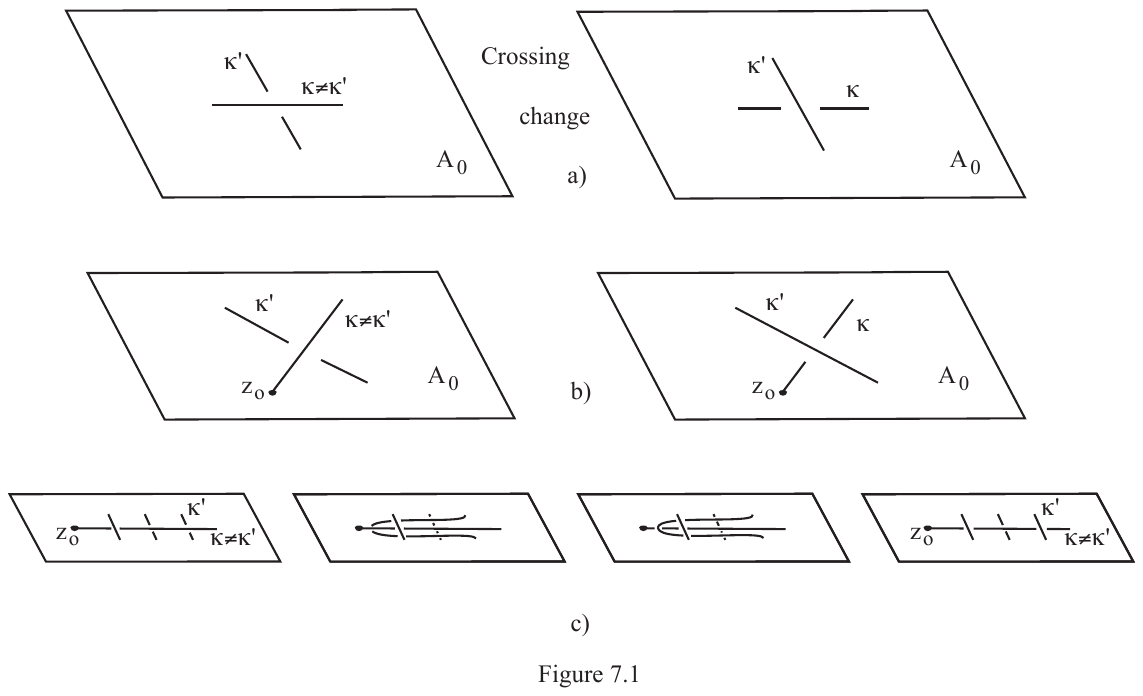}
\end{figure}

\begin{proof} Let $\kappa$ and $\kappa'$ denote either the distinct tube guide paths or the different components of $\alpha_i\setminus p_i$. Since by Lemma \ref{tube sliding lemma}  changing a tubed surface by type 2) and 3) Reidemeister moves does not change the isotopy class of the realization, it suffices to assume that the crossing is adjacent to $z_0$ as in Figure 7.1 b), e.g. see Figure 7.1 c).  Let $D_0\subset A_0$ be a small disc neighborhood of $z_0$ that contains the crossing and $D=f(D_0)$.  We can assume by Remark \ref{tubes unknotted} that there are no framed embedded paths in $N(G)=D\times G$.  We shall see that except  when Lemma \ref{light bulb trick} is invoked, the isotopy from $A$ to $A'$ is supported in $N(G)$.  Identify $N(G)$ with $D\times S^2=D\times S^1\times [-\infty,+\infty]/\sim$, where each $x\times S^1\times -\infty$ and each $x\times S^1\times +\infty$ is identified to a point. Also $f(z_0)=z=(0,t_0,0)$ and $D$ is identified with $D \times t_0 \times 0$ which we continue to call $D$.   By making $\kappa'$ very close to $z_0$ and the slightly pushed off copy of $G$ corresponding to $\kappa$ very close to $G$ we can  assume that $A\cap N(G)$ consists of exactly three components $D, K$ and $K'$ where the latter two respectively arise from $\kappa$ and $\kappa'$.
\begin{figure}[ht]
\includegraphics[scale=0.60]{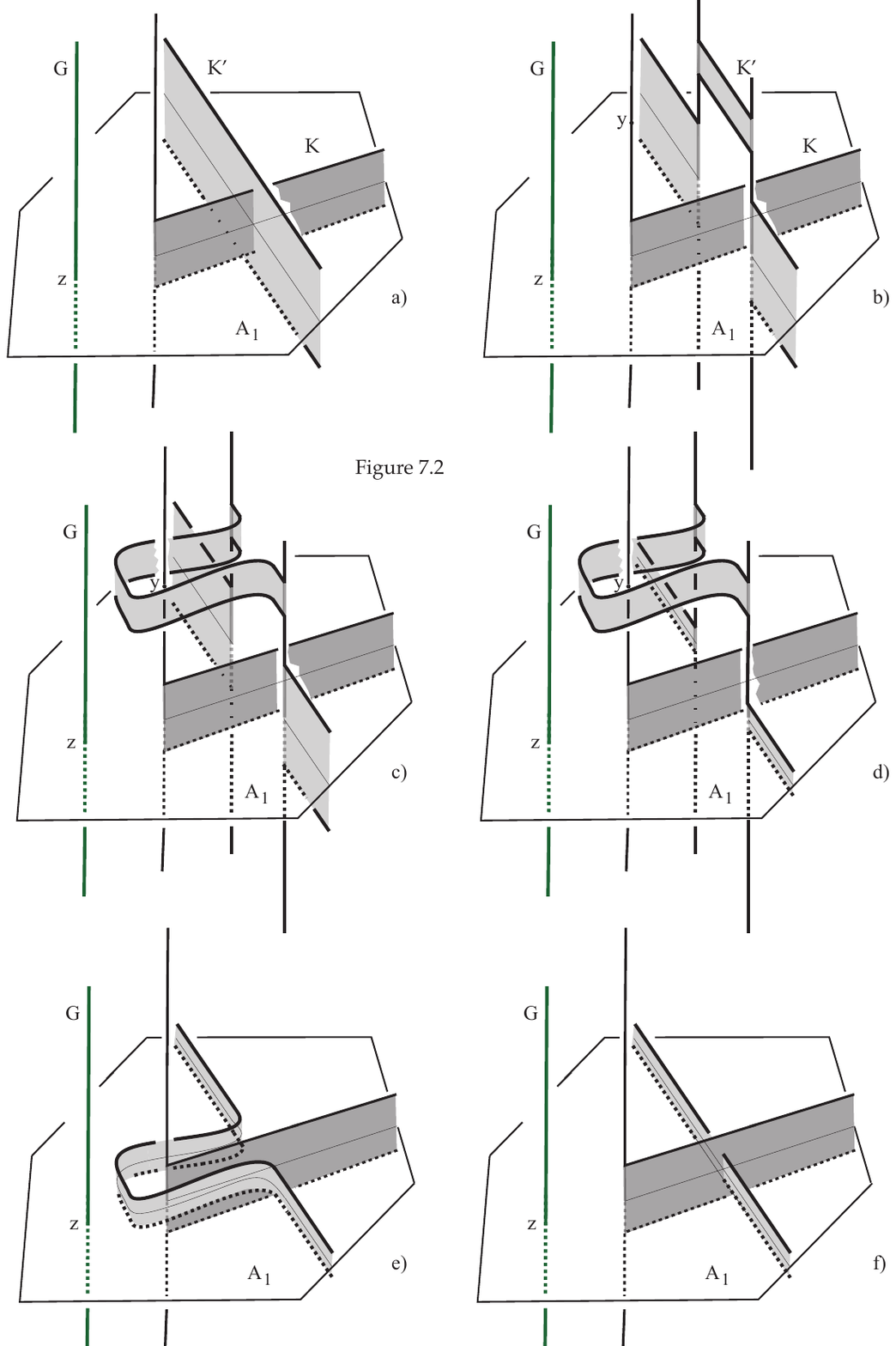}
\end{figure}

\begin{figure}[ht]
\includegraphics[scale=0.60]{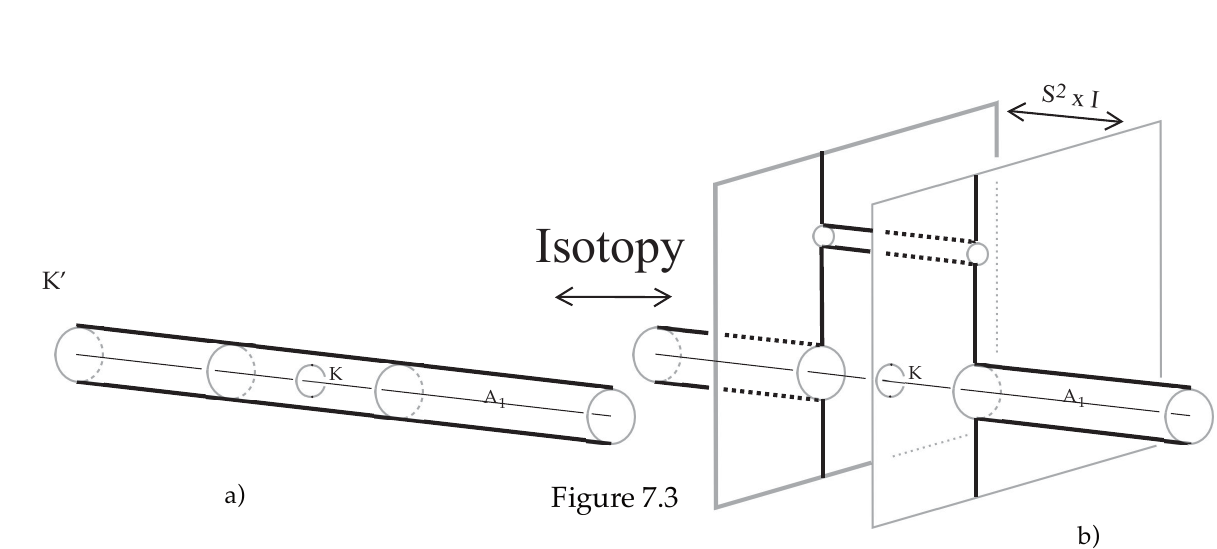}
\end{figure}

The isotopy from $A$ to $A'$  is demonstrated in Figure 7.2.  
The various subfigures  of Figure 7.2 show a 3-dimensional subset $H$ of the form $D\times [t_0-\epsilon, t_0+\epsilon] \times 0$ and the projection of $(G\cup A) \cap (H \times [-\infty,\infty]) $ to $H$, where the $A$ changes by isotopy as we progress from a) to f). Before offering more detail we decipher the figure.  Common to each  of a)-f) are  $D, G\cap H$ and the projection of $\hat K= K\cap (H \times [-\infty,\infty])$ to $H$.  Each figure also contains the projection of  $\hat K'$ to $H$, where $\hat K'=K'\cap (H \times [-\infty,\infty])$ and $K'\subset D\times G$. We abuse notation by calling $K'$ a surface that  changes during the progression of a) to f).  In each of a)-f) the projection of $K$ (resp. $K'$) to $D$ is an arc.   $K$ is obtained by attaching a tube $T_u$ to the sphere $u\times G$, where $u\in D$ and $T_u$ follows a path $t_u\subset D$ from $u$ to $\partial D$.  We again use the convention that thick lines indicate intersection with the present, i.e. $H$ and that the  shaded points of $T_u$ indicate those coming from the past or future.  In each of a)-f)  $K'$ is an annulus.  In a), e), f) it is a tube $T$ that follows a properly embedded arc in $D$.  To make $T$ and $T_u$ look more distinct we shaded them differently.  In Figures e) and f) (resp. a)) $T$ lies closer (resp. further) to $D$ than $T_u$.  Compare Figures a) and f) with Figure 5.3.  Finally, the $ K'$'s of b) and c) are obtained by attaching three tubes to the spheres $u_1\times G$ and $u_2\times G$.  One tube $T'$ connects the spheres to each other.  The other tubes respectively follow paths in $D$ from $u_1$ and $u_2$ to $\partial D$.  For all a)-f) each of $G, D, K$ and $K'$ are invariant under the reflection of $[-\infty,\infty]$ about $0$.

We return to the proof of the Lemma.  Figure 7.2 a) shows the projections of the components of  $A\cap H\times [-\infty,\infty]$ to $H$.    Let $\delta=f(\kappa')\cap D$.   The preimage $W$ of $\delta$ in $D\times S^2$ is an  $I\times S^2$ which intersects $A$ in three components; an annulus from $K'$, a $S^1$ from $K$ and the arc $\delta$, all of which are shown in the Figure 7.3 a).  From this point on, the support of the isotopy, in the source, is within $K'$.  The isotopy  from   Figures 7.2 a) to 7.2 b) is supported in $W$ and can be viewed in more detail in  Figure 7.3, where corresponding dark lines in Figures 7.2 a), b) and Figure 7.3 a), b) coincide.  The dots on $W\cap K$ in Figure 7.3 are the intersections of the dark lines on $K$ with $W$.  The passage from Figure 7.2 b) to Figure 7.2 c) is the 4-dimensional light bulb move, Lemma \ref{light bulb trick}, whereby the tube $T'$ appears to be crossing $A$ at the point $y$.  This  requires that there be a path $\sigma$ in $A$ from $y$ to $z$ disjoint from $T'$ which in turn requires that $\kappa\neq \kappa'$.  The isotopy from Figure 7.2 c) to d) simply squeezes the indicated tubes, and commutes with the previous one.  The isotopy corresponding to Figures 7.2 d) and e) is essentially the reverse of that from 7.2 a) and 7.2 b).   Here the projection of the $S^2\times I$, corresponding to this isotopy, to $D$ is an arc disjoint from the projection of $K$ to $D$, so  $K$ is not in the way during that isotopy. \end{proof}

\section{Proof of Theorems \ref{main} and  \ref{2-torsion}}

\begin{figure}[ht]
\includegraphics[scale=0.60]{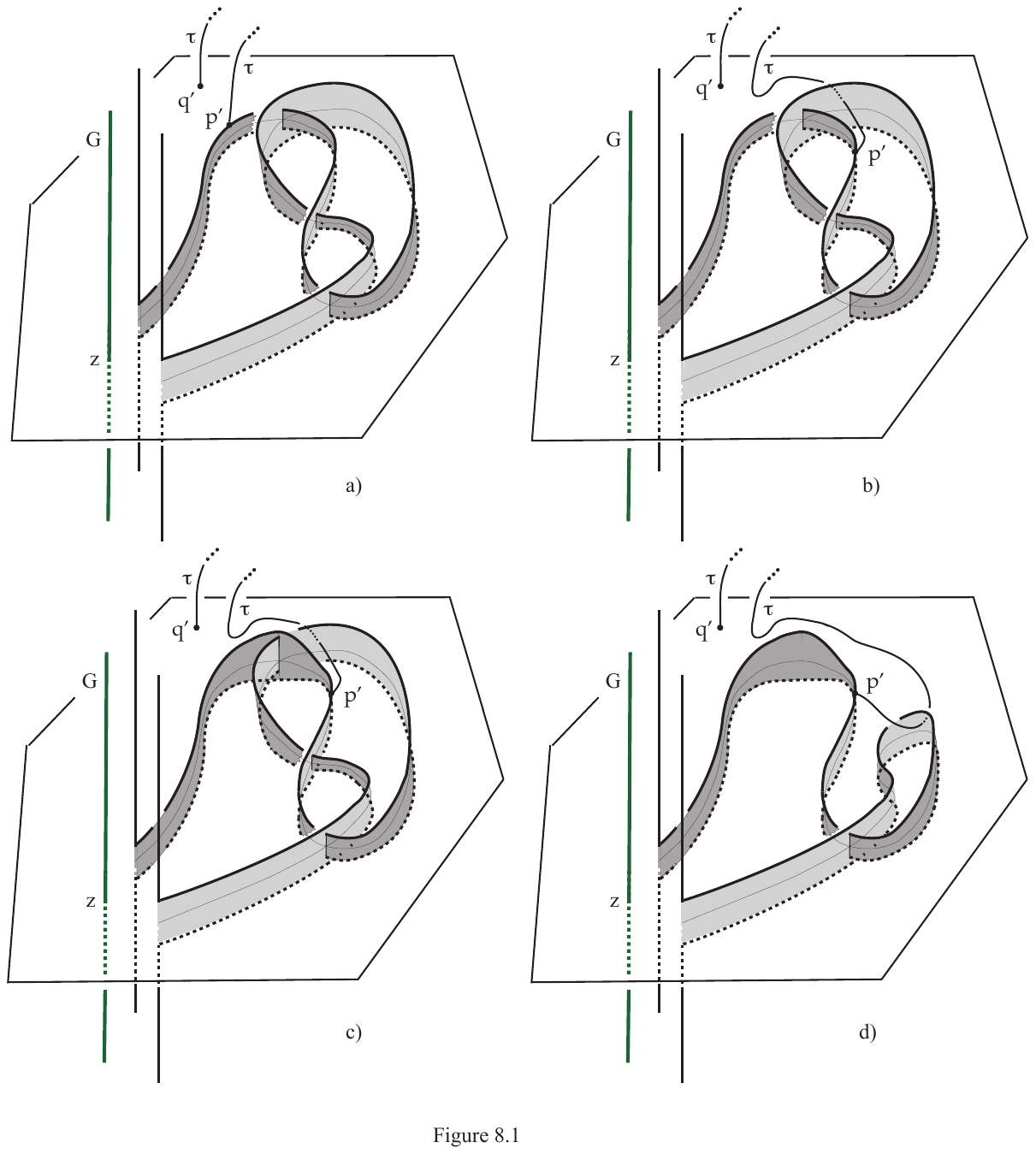}
\end{figure}

\begin{figure}[ht]
\includegraphics[scale=0.60]{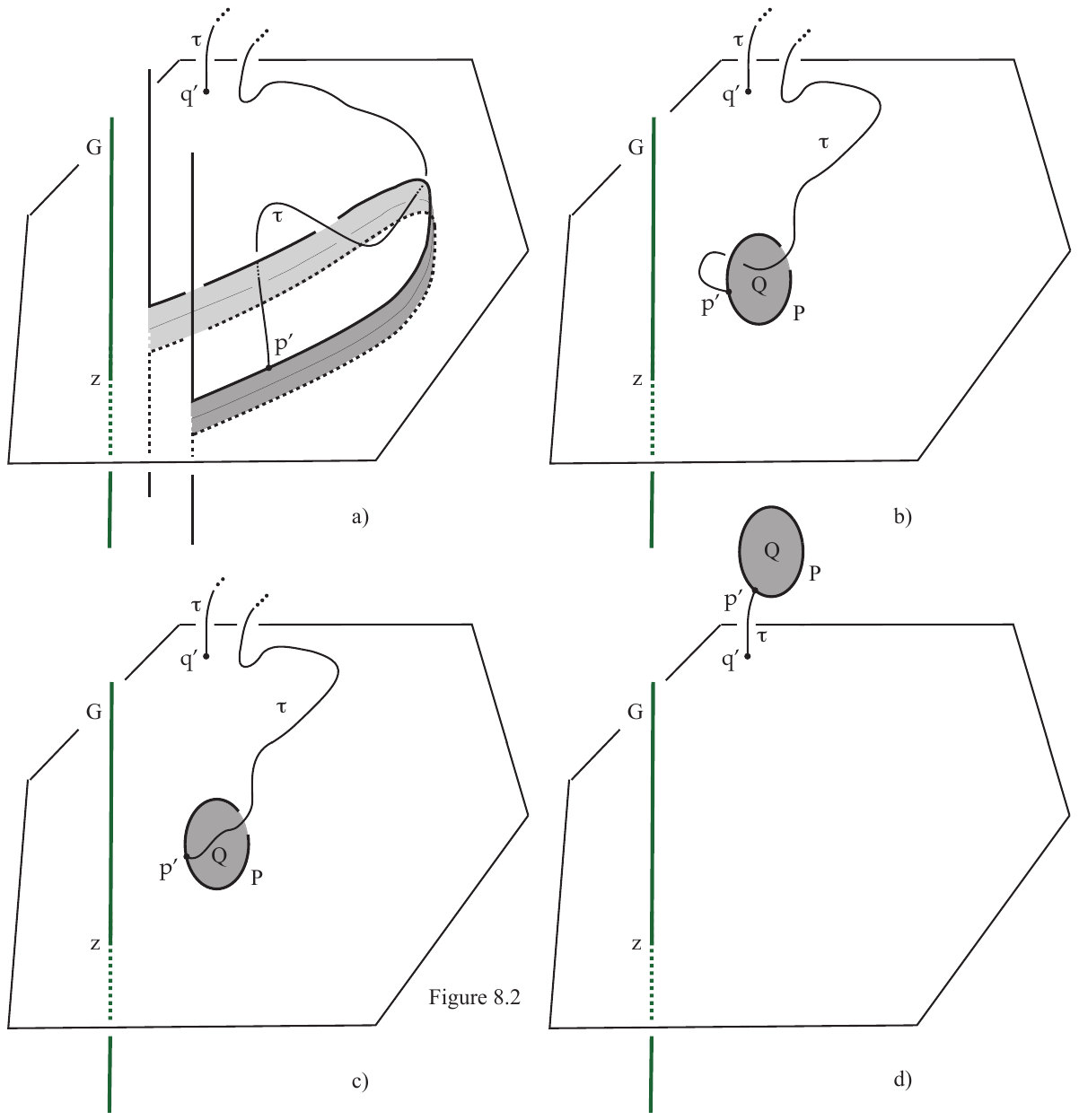}
\end{figure}

Suppose that the embedded spheres $R$ and $A_1$ are homotopic with the common transverse sphere $G$.  After an initial isotopy fixing $G$ setwise, we can assume that they coincide near $G$.  By Lemma \ref{basepoint} we can assume that  the homotopy from $R$ to $A_1$ is supported away from a neighborhood of $G$.   It follows from \cite{Sm1} that $R$ is regularly homotopic to $A_1$ via a homotopy that is also supported away from a neighborhood of $G$.  See \S4.  By Theorem \ref{shadow isotopy} the homotopy from $R$ to $A_1$ is shadowed by tubed surfaces.  Thus there exists a tubed surface $\mA$ with realization $A$, underlying surface $A_0$ and associated surface $A_1=f(A_0)$  such that $R$ is isotopic to $A$.  

To prove Theorem \ref{main} it suffices to show that $A$ is isotopic to $A_1$ via an isotopy supported away from a neighborhood of $G$.  To prove Theorem \ref{2-torsion} it suffices to show that $A$ can be isotoped into normal form with respect to $A_1$ and two surfaces $A$ and $A'$ in normal form with respect to $A_1$ are isotopic with support away from $G$ if their framed tube guide curves represent the same set of 2-torsion elements. By Proposition \ref{all realized} all finite sets of nontrivial 2-torsion elements can be realized.

By Proposition  \ref{double to single} we can assume that $A$ has finitely many double tubes, each representing distinct 2-torsion elements of $\pi_1(M)$  plus at most one double tube  which is homotopically trivial.  
  In what follows we assume for consistency of notation that there exists one such double tube and it's associated to $(\beta_0, \gamma_0)$.  Using the crossing change and tube sliding Lemmas \ref{crossing} and \ref{tube sliding lemma} we can assume that each pair $(\alpha_1, q_1), \cdots, (\alpha_r,q_r), (\beta_0, \gamma_0), (\beta_1, \gamma_1), \cdots, (\beta_n, \gamma_n)$ lies in a distinct sector of $A_0$.  This means there exists a neighborhood $D_0$ of $z_0\in A_0$ parametrized as the unit disc in polar coordinates so that $(\alpha_i, q_i)$ lies in the subset of $D_0$ where  $((i-1)/(r+n+1))2\pi<\theta<(i/(r+n+1))2\pi$ and $(\beta_i,\gamma_i)$ lies in the region $((r+i)/(r+n+1))2\pi<\theta<((r+i+1)/(r+n+1))2\pi$.  We can further assume that $q_1\in \partial D_0$.  This requires that $R_0$ is a 2-sphere. 
  
After $r$  applications of the next lemma we can assume that $r=0$ and that the remaining data defining $\mA$ is unchanged.
  
\begin{lemma}  \label{local unknotting}  Let $\mA$ be a tubed surface such that a neighborhood $D_0\subset A_0$ of $z_0$ is parametrized as the unit disc  in polar coordinates and $\alpha_1 \cup q_1 $ is contained in the subset of $D_0$ where $0<\theta<\pi/2$ and that that region is devoid of all other tube guide curves and associated points.  Assume also $q_1\in \partial D_0$.  Then the tubed surface $\mA'$  whose data consists of that of $\mA$ with $\alpha_1\cup \tau_1\cup q_1$ deleted has realization $A'$ isotopic to the realization $A$ of $\mA.$\end{lemma}


\begin{proof}  We first fix some terminology.  To simplify notation $\alpha_1, p_1, q_1, \tau_1$ will be respectively denoted $\alpha, p,q,\tau$. $D=f(D_0)$ and  $T(\alpha)$ will denote the tube about $f(\alpha)$. $P(\alpha)$ will denote the 2-sphere consisting of two parallel copies of $G$ tubed together along $T(\alpha)$.   $T(\tau)$ will denote the tube about $\tau$.   So $A$ is the surface obtained from $A'$ by connecting $P(\alpha)$ to $A'$ by the tube $T(\tau)$.  We will let $p'$ and $q'$ denote the points respectively on $P(\alpha)$ and $A'$ so that the ends of $T(\tau)$ connect to $\partial N(p')$ and $\partial N(q')$, where neighborhoods are taken in $A$.  Here $p'$ orthogonally projects to $f(p)\in A_1$, where $p\in \alpha\subset A_0$, and $q'=f(q)$.  We let $\alpha_L$, and $\alpha_R$ denote the components of $\alpha$ separated by $p$.

Before continuing we offer commentary on Figure 8.1 a).  That figure shows the 3-dimensional slice $D\times [t_0-\epsilon, t_0+\epsilon]\times 0\subset N(D)=D\times [t_0-\epsilon, t_0+\epsilon]\times [-\epsilon,\epsilon]$. Let $\pi:N(D)\to D\times [t_0-\epsilon, t_0+\epsilon]\times 0$ be the projection. Shown are $\pi(G\cap N(D))$, $\pi(P(\alpha)\cap N(D))$ and $\pi(\tau\cap N(D))$.  Here $\pi(\tau\cap N(D))\subset  D\times [t_0-\epsilon, t_0+\epsilon]\times 0$.  Note that $\pi(P(\alpha)\cap N(D))$ consists of two vertical arcs together with an immersed strip connecting them.  The cross sectional height of the strip indicates how far or close the corresponding tube is from $D$, i.e. away from where $T(\alpha)$ joins copies of $G$, a disc transverse to $D$ at $x\in f(\alpha)$ intersects $T(\alpha)$ in a circle whose radius is equal to the height of the strip.  We follow the convention of Construction \ref{tube construction} illustrated in Figure 5.3 whereby an undercrossing of $\alpha$ corresponds to a segment of a strip that is thinner than  the other segment it intersects.  Again, thick lines denote the intersection of $P(\alpha)\cap N(D)$ with the present and points are shaded if they come from the past or future.  The different shading of the strip is for visual purposes only, to help see it as it twists and turns.

The first observation is that by isotopy extension $p'$ can be isotoped to any point in $T(\alpha)$ at the cost of seemingly \emph{entangling} $\tau$ with $T(\alpha)$.  See Figures 8.1 a) and b).  One cannot obviously use the light bulb lemma to remove the intersection of  $\inte(\tau)$ with the projection of $T(\alpha)$ in Figure 8.1 b), since $T(\tau)$ separates $T(\alpha)$ from $z$.

By suitably moving $p$, the local arcs of $\alpha$ at a given crossing can lie in the different components $\alpha_L$ and $\alpha_R$ of $\alpha \setminus p$.    Thus the proof of the crossing change lemma allows us to change this crossing as well as any other at the cost of entangling $\tau$ with $T(\alpha)$.  This process is illustrated in Figures 8.1 a), b), c).  Similarly, at the cost of further such entanglement we can perform the reordering move, Definition \ref{tube sliding} ii) to arcs of $\alpha$.  Compare Figure 8.1 d) and Figure 8.2 a).  Thus by crossing changes, Reidemeister 2), 3) moves and the tube sliding reordering move, we can assume that $\alpha$ has no crossings.     It follows that $P(\alpha)$ can be isotoped to an unknotted 2-sphere $P$, that bounds a 3-ball $Q$ disjoint from $A'\cup G$.  See Figure 8.2 b).  Further there exists a 4-ball $B$ such that $Q\subset B$, $A'\cap B=\emptyset$, $B\cap G=\emptyset$ and $\tau\cap B$ is connected.  While it can be avoided, we note that any entanglement with $A_1$ can be eliminated by Lemma \ref{light bulb trick}.  Since $\pi_1(B\setminus P)= \BZ$  and $T(\tau)$ can be rotated about $P$ it follows that via isotopy supported within $B$, $A$ can be isotoped so that $Q\cap \inte(T(\tau))=\emptyset$ as in Figure 8.2 c).  It follows that $A$ can be isotoped to $A'$, thereby completing the proof of Lemma \ref{local unknotting}.\end{proof}

We now assume that $r=0$.  Let $\mA'$ be the tubed surface obtained from $\mA$ with the data $(\beta_0,\gamma_0)$ and $\lambda_0$ deleted.  We now show that the realization $A$ of $\mA$ is isotopic to the realization $A'$ of $\mA'$.  Since $\lambda_0$ can be homotoped rel endpoints into the associated surface $A_1$, Lemma \ref{pi injectivity} and Remark \ref{tubes unknotted} imply that $\lambda_0$  can be  isotoped to lie near the sector of $f(D_0)$ containing $f(\beta_0\cup \gamma_0)$, via an isotopy supported away from $G\cup A_1$.  Hence via isotopy supported away from $G$, $A$ can be isotoped so that the double tube following $\lambda_0$ and $D(\beta_0)\cup D(\gamma_0)$ lie   in a small regular neighborhood  $N(G')$ of a parallel translate $G'$ of $G$.  Here $N(G')=G'\times D^2$, where $A_1\cap N(G') = z'\times D^2$ and $N(G')\cap G=\emptyset$.  Also $A$  coincides with $A'$ outside of $G'\times D^2$,  $A$ coincides with both $A'$ and $A_1$ near $\partial N(G')$ and  $[A\cap N(G')]=[z'\times D^2] \in H_2(N(G'),z'\times \partial D^2)$.  By  Theorem \ref{uniqueness} $A$ is isotopic to $A'$ via an isotopy supported within $N(G')$.    This completes the proof of Theorem \ref{main}.  \qed
\begin{remark} \label{dependence}    With more work one can eliminate reliance on Theorem \ref{uniqueness}.  By Lemma \ref{crossing} and \ref{tube sliding lemma} we can assume that $\beta_0$ and $\gamma_0$ are pairwise disjoint.  If each is embedded, then a direct argument allows for the elimination of this data from $\mA$, via isotopy of $A$.  If either $\beta_0$ or $\gamma_0$ are not embedded, then construct $\mA'$ whose data equals that of $\mA$ except that an adjacent sector contains  $\beta_{-1}, \gamma_{-1}$ which are disjoint and embedded and whose framed embedded path $\lambda_{-1}$ is homotopically trivial.   As noted above the realizations $A $ and $A'$ of $\mA$ and $\mA'$ are isotopic by a direct argument.   The proof of Proposition \ref{double to single} shows that the  pair of double tubes associated to $\lambda_0$ and $\lambda_{-1}$ can be transformed into a pair of single tubes and these tubes and their associated data can be eliminated via isotopy as in the  proof of Theorem \ref{main} without using Theorem \ref{uniqueness}.  \end{remark}
\begin{figure}[ht]
\includegraphics[scale=0.60]{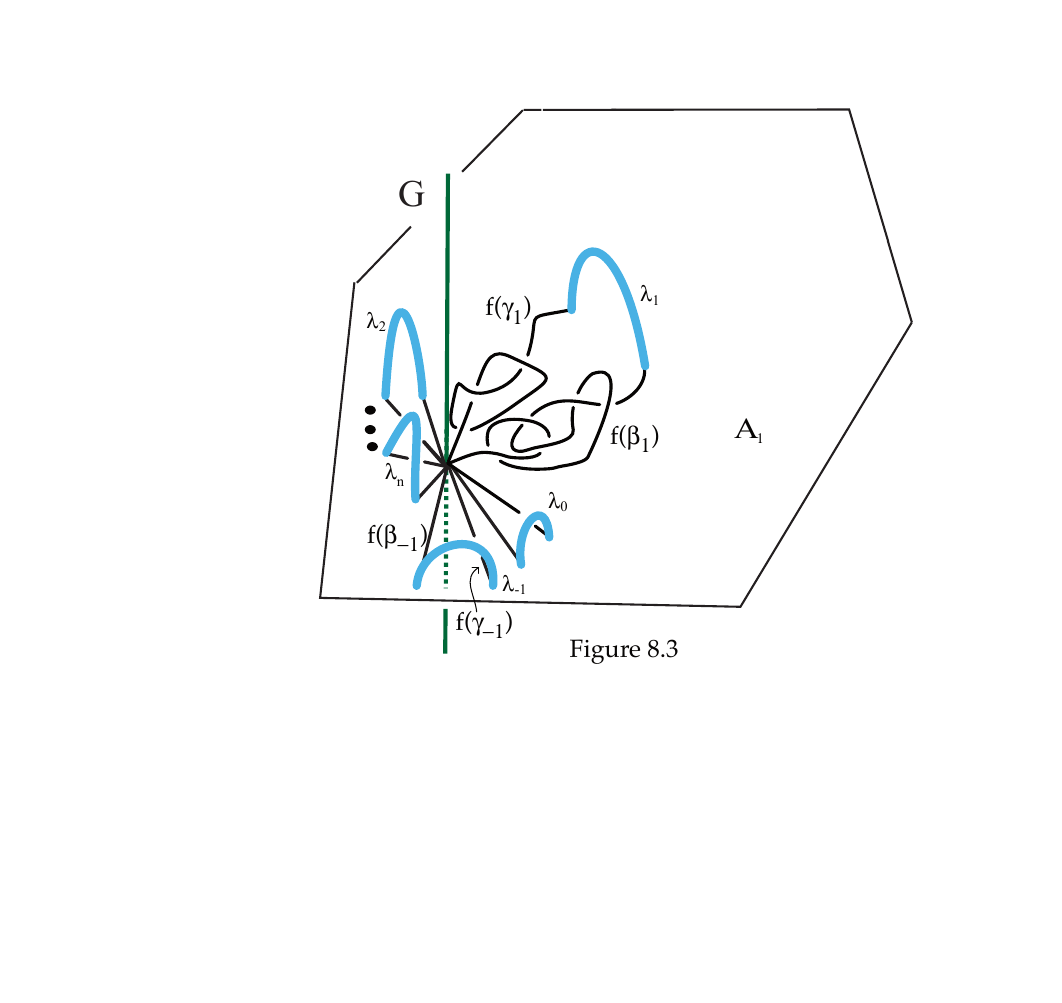}
\end{figure} 
We now complete the proof of Theorem \ref{2-torsion}.  We have a tubed surface $\mA$ whose data consists of $(\beta_1,\gamma_1) \cdots, (\beta_n,\gamma_n)$ and  framed tube guide paths $\lambda_1, \cdots, \lambda_n$ representing distinct nontrivial 2-torsion elements of $\pi_1(M)$.  As above we can assume that the $(\beta_i,\gamma_i)$'s lie in distinct sectors of $D_0\subset A_0$.  Let $\mA'$ be a tubed surface with data $(\beta'_1,\gamma'_1), \cdots, (\beta'_n,\gamma'_n)$ each lying in distinct sectors and where $\beta'_i,\gamma'_i$ are disjoint and embedded.  Further each $\lambda'_i$ is homotopic to $\lambda_i$, though its framing is allowed to be arbitrary.  

We now show that the realizations $A$ and $A'$ of $\mA$ and $\mA'$ are isotopic.  Using the idea of the previous remark we show how to replace $\beta_1,\gamma_1, \lambda_1$ by $\beta'_1, \gamma'_1$ and $\lambda'_1$ without changing the isotopy class of the realization.  The proof that $A$ and $A'$ are isotopic then follows by induction.   First observe that  there are adjacent sectors of $D_0$ containing the ray of angle 0, which are disjoint from $(\beta_n,\gamma_n)$ and $(\beta_1,\gamma_1)$.   Let $\mA''$ be the tubed surface whose data consists of that of $\mA$ with $(\beta_{-1},\gamma_{-1})$ and $(\beta_0, \gamma_0)$ added to these sectors where for $i\in\{-1,0\}$,  $\beta_i$ and $\gamma_i$ are disjoint and embedded.  Further the framed tube guide curves $\lambda_{-1}$ and $\lambda_0$ represent the homotopy class of $\lambda_1$, with the framing given by $\lambda'_1$.   See Figure 8.3.  Now the underlying surface $A''$ of $\mA''$ is isotopic to $A$ since the proof of Proposition \ref{double to single} shows how to transform the pair of double tubes associated to $\lambda_{-1}$ and $\lambda_0$ to single tubes and these single tubes can be eliminated as usual.   Similarly $A''$ is isotopic to $A'$ since the pair of double tubes associated with $\lambda_0$ and $\lambda_1$ can be transformed into a pair of single tubes.   These single tubes can then be eliminated as usual.   Further,  $(\beta_{-1}, \gamma_{-1})$ can now be rotated into the position of $(\beta'_1, \gamma'_1)$ and $\lambda_{-1}$ can be isotoped to $\lambda'_1$.  By construction $\lambda_{-1}$ has the same framing as that of $\lambda'_{1}$. It follows that $A$ is isotopic to $A'$.  


In summary, if $R$ and $A$ are embedded spheres in $M$ with the common transverse sphere $G$, then $R$ can be isotoped into normal form with respect to $A$ as in Definition \ref{normal form}.    If $A'$ and $A''$ are in normal form with respect to $A$ and represent the same set of elements, then Lemma \ref{normal permutation} shows that after isotopy we can assume that for all $i, \lambda'_i$ is homotopic to $\lambda''_i$.  Thus after isotopy we can assume that $A'$ and $A''$ are realizations of tubed surfaces $\mA'$ and $\mA''$ such that for all $i$, $(\alpha'_i, \beta'_i)=(\alpha''_i, \beta''_i)$ and $\lambda'_i=\lambda''_i$, however the framings of the $\lambda_i$'s might differ.   The previous paragraph shows that changing the framing of a $\lambda''_i$ does not change the isotopy class of the realization.  It follows that $A'$ and $ A''$  are isotopic and hence equivalent normal forms are isotopic. By Proposition \ref{all realized} any finite set of distinct nontrivial 2-torsion elements can be represented by a surface in normal form with respect to $A$.  This completes the proof of Theorem \ref{2-torsion}. \qed

\section{higher genus surfaces}

In this section we give a partial generalization of our main result to higher genus surfaces, that is a full generalization for $\stwostwo$.  

\begin{definition} Let $S$ be an immersed surface in the 4-manifold $M$.  We say that the embedded disc $D\subset M$ is a \emph{compressing disc} for $S$ if $\partial D\subset S$ and a section of the normal bundle to $\partial D\subset S$ extends to a section of the normal bundle of $D\subset M$. \end{definition}

\begin{lemma} \label{compressing} If $S$ is immersed in the 4-manifold $M$ and $\alpha\subset S$ is an embedded curve with a trivial normal bundle in $S$ and  is homotopically trivial in $M$, then $\alpha$ bounds a compressing disc.  \end{lemma}

\begin{proof}  First span $\alpha$ by an immersed disc $D_0$.  Using boundary twisting \cite{FQ} we can replace $D_0$ by $D_1$ that satisfies the normal bundle condition.  Eliminate the self intersections of $D_1$ by applying finger moves.\end{proof}

\begin{lemma} Let $S$ be an orientable embedded surface in the 4-manifold $M$ whose components have pairwise disjoint transverse spheres.  Let $\alpha_1, \cdots, \alpha_k \subset S$ be pairwise disjoint simple closed curves, disjoint from the transverse spheres, such that for each component $S'$ of $S$, $S'\setminus \{\alpha_1,\cdots, \alpha_k\}$ is connected.  Suppose that for each $i$, $\alpha_i$ is homotopically trivial in the complement of the transverse spheres.  Then there exist pairwise disjoint compressing discs $D_1, \cdots, D_k$ such that for each $i$, $D_i\cap S=\alpha_i$.  \end{lemma}

\begin{proof}  Construct compressing discs $A_1, \cdots, A_k$ for the $\alpha_i$'s as in Lemma \ref{compressing}.  These discs can be chosen to be disjoint from the transverse spheres.  Using finger moves they can be made disjoint from each other.  Finally use the transverse spheres to tube off unwanted intersections of the $A_i$'s with $S$ to create the desired $D_i$'s.\end{proof}

\begin{definition} We say that the surface $S_1$ is obtained from $S$ by \emph{compressing} along $D$ if $S_1 = S\setminus \inte(N(\partial D)) \cup D'\cup D''$ where $D',D''$ are two pairwise disjoint parallel copies of $D$.\end{definition}

\begin{lemma}  Surfaces can be compressed along compressing discs.  If $S_1$ is obtained by compressing the embedded surface $S$ along the compressing disc $D$ and $D\cap S=\partial D$, then  $S_1$ is embedded.\qed\end{lemma}

\begin{definition} We say that the surface $S\subset M$ is $G$-\emph{inessential} if the induced map $\pi_1(S\setminus G)\to \pi_1(M\setminus G)$ is trivial.\end{definition}

The following is a generalization to higher genus surfaces of Theorem \ref{main}.

\begin{theorem}  \label{genus} Let $M$ be an orientable 4-manifold such that $\pi_1(M)$ has no 2-torsion.  Two homotopic, embedded, $G$-inessential surfaces $S_1, S_2$ with common transverse sphere $G$  are ambiently isotopic.  If they coincide near $G$, then the isotopy can be chosen to fix a neighborhood of $G$ pointwise. \end{theorem}

\begin{proof}  For each $i\in \{1,2\}$ let $\alpha_1^i, \cdots, \alpha_k^i$ be a set of pairwise disjoint simple closed curves in $S_i$ whose complement is a connected planar surface containing $S_i\cap G$.  Let $D_1^i, \cdots, D_k^i$ be associated pairwise disjoint compressing discs with interiors disjoint from $S_i$ and let $T_i$ be the result of compressing $S_i$ along these discs.  Then $T_i$ is a 2-sphere and $S_i$ is obtained from $T_i$ by attaching $k$ tubes.  Each tube $S^1\times I$ extends to a solid tube $D^2\times I$ which intersects $T_i$ exactly at $D^2\times 0$ and $D^2\times 1$, which we call the \emph{bases of the tube}.  By construction, these tubes are pairwise disjoint.  After a further isotopy we can assume there are $k$ small pairwise disjoint 4-balls each of which intersects $S_i$ in a single standard disc and each such disc contains the bases of a single solid tube.

Since $S_i  $ is $G$-inessential, it follows by the light bulb lemma that the solid tubes can be isotoped to 3-dimensional neighborhoods of tiny standard arcs with endpoints on $T_i$.   Note that the induced ambient isotopy can be chosen to  fix a neighborhood of $T_i$ pointwise.  

To complete the proof it suffices to show that $T_1$ and $T_2$ are homotopic and hence isotopic by Theorem \ref{main}.  To see this, consider the lifts $\tilde T_1, \tilde T_2$  of $T_1,T_2$ to the universal covering $\tilde M$ of $M$ which intersect a given lift $\tilde G$ of $G$.  Since the $S_i$'s are $\pi_1$-inessential and homotopic, their corresponding lifts $\tilde S_1, \tilde S_2$ are homotopic and hence homologous.  It follows that $\tilde T_1$ and $\tilde T_2$ are homologous and hence homotopic and therefore so are $T_1$ and $T_2$.  \end{proof}  

Applying to the case of $\stwostwo$ we obtain:

\begin{theorem}  \label{higher genus} Let $R$ be a connected embedded genus-$g$ surface in $\stwostwo$ such that $R\cap  S^2\times y_0=1$.  Then $R$ is isotopically standard.  I.e. it is isotopic to the standard sphere  in its homology class, with $g$ standard handles attached, which we denote by $R_0$.  If $R$ and $R_0$ coincide near $S^2\times y_0$, then the isotopy can be chosen to  fix a neighborhood of  $S^2\times y_0$ pointwise. \qed\end{theorem}

\section{applications and questions}

We begin by stating the main result for multiple spheres.

\begin{theorem}\label{multi main}  Let $M$ be an orientable 4-manifold such that $\pi_1(M)$ has no 2-torsion.  Let $G_1, \cdots, G_n$ be pairwise disjoint embedded spheres with trivial normal bundles.  Let $R_1, \cdots, R_n$ be pairwise disjoint embedded spheres transverse to the $G_i$'s such that $|R_i\cap G_j|=\delta_{ij}$.  Let $S_1, \cdots, S_n$ be another set of spheres with the same properties and coinciding with the $R_i$'s near the $G_i$'s. If for each $i$,\ $R_i$ is homotopic to $S_i$, then there exists an isotopy of $M$ fixing a neighborhood of the $G_i$'s pointwise such that for all $j$, $R_j$ is taken to $S_j$.

Under corresponding hypotheses, the same conclusion holds when the $R_i$'s and $S_i$'s are $G$-inessential connected surfaces, where $G=\cup_{i=1}^n G_i$.     \end{theorem}

\begin{proof} The methods of \S 9 reduce the general case to the case that all the $R_i$'s and $S_i$'s are spheres.  \vskip 8pt

\noindent Proof by induction on $n$. 
\vskip8pt 

\noindent\emph{Step 1:} $R_1$ is ambient isotopic to $S_1$ via an isotopy that fixes the $G_i$'s pointwise. \vskip 8pt

\noindent\emph{Proof}  After a preliminary isotopy we can assume that $R_1$ and $S_1$ coincide near $R_1\cap G$ and that the homotopy from $R_1$ to $S_1$ is supported away from a neighborhood of $\cup G_i$.  
Step 1 follows by applying Theorem \ref{main} to the manifold $M\setminus\cup_{i=2}^n N(G_i)$.   Note that the inclusion of $M\setminus\cup_{i=2}^n N(G_i)\to M$ induces a fundamental group isomorphism so the no 2-torsion condition is satisfied.    \qed
\vskip 8pt

\noindent\emph{Induction Step}:  Suppose that we have for $j<k$,\ $R_j = S_j$.  There exists an isotopy of $M$ fixing $\cup_{i=k}^n G_i$ pointwise and supported away from $\cup_{j=1}^{k-1}( G_j\cup S_j)$ such that $R_k$ is taken to $S_k$. 
\vskip 8pt
\noindent\emph{Proof} After a preliminary isotopy we can assume that $R_k$ and $S_k$ coincide near $G_k$ and that $R_k$ is homotopic to $S_k$ via a homotopy supported away from $\cup_{j=1}^{k-1} (S_j\cup G_j)$.  Next apply Step 1 to $R_k\subset M\setminus N(\cup_{j=1}^{k-1} (S_j \cup G_j))$.  Again, the argument of Lemma \ref{pi injectivity} implies that the inclusion $M\setminus N(\cup_{j=1}^{k-1} (S_j \cup G_j))\to M$ induces a fundamental group isomorphism, so the no 2-torsion condition is satisfied.  \end{proof}

An analogous argument combined with the proof of Theorem \ref{2-torsion} yields the corresponding result for manifolds with 2-torsion in their fundamental groups.

\begin{theorem}\label{multi 2-torsion}  Let $M$ be an orientable 4-manifold.  Let $G_1, \cdots, G_n$ be pairwise disjoint embedded spheres with trivial normal bundles.  Let $R_1, \cdots, R_n$ be pairwise disjoint embedded spheres transverse to the $G_i$'s such that $|R_i\cap G_j|=\delta_{ij}$.  Let $S_1, \cdots, S_n$ be another set of spheres with the same properties and coinciding with the $R_i$'s near the $G_i$'s. If for each $i$,\ $R_i$ is homotopic to $S_i$, then there exists an isotopy of $M$ fixing a neighborhood of the $G_i$'s pointwise such that each $R_i$ can be put into normal form with respect to $S_i$.   Here each $R_i$ has double tubes representing elements $\{[\lambda^i_1], \cdots, [\lambda^i_{n_i}]\}$, where for fixed $i$ the $[\lambda^i_j]$'s are distinct nontrivial 2-torsion elements of $\pi_1(M)$ and $R_i=S_i$ if this set is empty.  Finally, for each $i$ any finite set of distinct 2-torsion elements gives rise to an $R_i$ in normal form with double tubes representing this set and two sets of $R_i$'s are isotopic if  their corresponding sets of $[\lambda^i_j]$'s are pairwise equal. \qed \end{theorem}

\begin{definition} An essential simple closed curve in $S^2\times S^1$ is said to be \emph{standard} if it is isotopic to $x\times S^1$ for some $x\in S^2$. \end{definition}

\begin{theorem}  \label{boundaries coincide}Two properly embedded discs $D_0$ and $D_1$ in $S^2\times D^2$ that coincide near their standard boundaries are isotopic rel boundary if and only if they are homologous in  $H_2(S^2\times D^2, \partial D_0)$.\end{theorem}

\begin{proof}  Homologous is certainly a necessary condition.  In the other direction, after reparameterizing, we can assume that $\partial D_0=x_0\times S^1\subset S^2\times S^1$. Let $M=S^2\times D^2\cup d(S^2\times D^2)=S^2\times S^2$ be obtained by doubling $S^2\times D^2$ with $d(S^2\times D^2)$ denoting the other $S^2\times D^2$.   This $d(S^2\times D^2)$ can be viewed as a regular neighborhood $N(G)$ of $G=d(S^2\times 0)$.    Let  $R_i$ denote the sphere  $D_i\cup d(x_0\times D^2)$ which we can assume is smooth for $i=0,1$.  $G$ is a transverse sphere to the homologous spheres $R_0$ and $R_1$.  By Theorem \ref{main} there is an  isotopy of $M$ fixing a neighborhood of $ G$ pointwise taking $R_0$ to $R_1$.  Since $R_0$ and $R_1$ coincide in a  neighborhood of $N(G)$ there is  an  isotopy of $S^2\times D^2$ taking $D_0$ to $D_1$ that fixes a neighborhood of $S^2\times S^1$ pointwise.  \end{proof}

\begin{theorem} \label{disc unknotted} A properly embedded disc $D$ in $S^2\times D^2$ is properly isotopic to a $D^2$-fiber if and only if its boundary is isotopic to the standard vertical curve.\end{theorem}

\begin{proof} After a preliminary isotopy we can assume that $\partial D$ is the standard vertical curve $x_0\times S^1$ which we denote by $J$.  Let $F$ be a $D^2$ fiber of $S^2\times D^2$.  Now $0\to H_2(S^2\times D^2)\to H_2(S^2\times D^2,  J)\to H_1(J)\to 0$ is split and exact, so the coset $H$ mapping to the generator $[\partial F]$ of $H_1(J)$ equals $Z$ and is represented by the classes $[F]+n[S^2\times y_0]$, where $y_0\in \partial D^2$.  By properly isotoping $D$ to $D'$ where $\partial D'=J$ and so that the track of the homotopy restricted to the boundary is approximately $J\cup S^2\times y_0$ it follows that $[D']=[D]+[S^2\times y_0]\in H_2(S^2\times D^2,  J)$.  Therefore any class in $H$ is represented by a disc properly isotopic to $D$.  In particular after proper isotopy we can assume that $[D]=[F]$.  After a further isotopy we can assume that $D$ coincides with $F$ near $\partial D$.  The result now follows by Theorem \ref{boundaries coincide}.  The other direction is immediate.\end{proof}

Recall that $\Diff_0(X)$ denotes the group of diffeomorphisms properly homotopic to the identity.

\begin{corollary}  $\pi_0(\Diff_0(S^2\times D^2)/\Diff_0(B^4))=1$.\end{corollary}

\begin{remark}  This means that a diffeomorphism of $S^2\times D^2$ properly homotopic to the identity is isotopic to one that coincides with the identity away from a compact 4-ball disjoint from $S^2\times S^1$.\end{remark}

\begin{proof}  Let $D$ denote a $x\times D^2$ and let  $f:S^2\times D^2\to S^2\times D^2$ be properly homotopic to the identity.  Since homotopic diffeormorphisms of $S^2\times S^1$ are isotopic \cite{La}, $f(\partial D)$ is isotopic to $\partial D$ in $S^2\times S^1$ and hence isotopically standard.   Next apply Theorem \ref{disc unknotted} to isotope $f$ so that $f(D)=D$.  After a further isotopy, using \cite{Sm3}, we can assume that $f|D=$ id and after another that $f|N(D)=$ id. Since a diffeomorphism of $D^2\times S^1$ that fixes a neighborhood of the boundary pointwise is isotopic to the identity rel $\partial$ it follows that we can be further isotope  $f$ so that $\partial f$ is also the identity.  After another isotopy we can additionally assume that $f|N(\partial (S^2\times D^2))= \id$.     
Since the closure of what's left is a $B^4$, the result follows.\end{proof}

The following is an immediate consequence of our main result.

\begin{theorem}  (\textrm{4D-Lightbulb Theorem}) If $R$ is an embedded 2-sphere in $S^2\times S^2$, homologous to $x_0\times S^2$,  that intersects $S^2\times y_0$ transversely and only at the point $(x_0, y_0)$, then $R$ is isotopic to $x_0\times S^2$ 
via an isotopy fixing $S^2\times y_0$ pointwise.\qed\end{theorem}

Litherland \cite{Li} proved that there exists a diffeomorphism pseudo-isotopic to the identity that takes $R$ to $x_0\times S^2$.  

Another version of the light bulb theorem was obtained in 1986 for PL discs in $S^4$ by Marumoto \cite{Ma} where the isotopy is topological.  He makes essential use of Alexander's theorem that any homeomorphism of $B^n$ that is the identity on $S^{n-1}$ is (topologically) isotopic to the identity.  Here we prove a general form of  the smooth version. 


\begin{theorem}  (Uniqueness of Spanning Surfaces) \label{marumoto} If $R_0$ and $R_1$ are smooth embedded surfaces in $S^4 $ of the same genus  such that $\partial R_0 = \partial R_1=\gamma$, where $\gamma$ is connected, then there exists a smooth isotopy of $S^4$ taking $R_0$ to $R_1$ that fixes $\gamma$ pointwise.\end{theorem}

\begin{proof}  First consider the case that $R_0$ and $R_1$ are discs.  After a preliminary isotopy of $S^4$ that fixes  $\gamma$ pointwise, we can assume that $R_0$ and  $R_1$ coincide in an annular neighborhood of their boundaries.  Now $S^4\setminus\inte(N(\gamma))=S^2\times D^2$.  Thus $R_0$ and $R_1$ restrict to properly embedded discs $E_0$ and $E_1$ in $S^2\times D^2$ that coincide near their boundaries.  

Arguing as in the proof of Theorem \ref{disc unknotted} we can assume that after an isotopy of $R_0$, $[E_0]=[E_1]\in H_2(S^2\times D^2, \gamma)$ also holds.  This isotopy fixes $\partial R_0$ pointwise but moves it's annular neighborhood.   Here are more details.  Let $\alpha_0$  denote $\partial E_0\subset S^2\times S^1$.  Let $A_0$ be the annulus bounded by  $\alpha_0$ and $\partial R_0$.   An isotopy of $\alpha_0$ induces an isotopy of $A_0$ fixing  $\gamma$ pointwise that extends to $R_0$ by isotopy extension.  The resulting $R'_0$ has annular boundary coinciding with that of $R_1$ and the class of the resulting $E'_0$ in the coset $H\subset H_2(S^2\times D^2, \gamma)$ changes according to the number of times the isotopy of $\alpha_0$ algebraically sweeps across the $S^2$-factor.  Here $H$ is as in the proof of Theorem \ref{disc unknotted}.

It follows by Theorem \ref{boundaries coincide} that $E_0$ can be isotoped to $E_1$ via an isotopy supported away from a neighborhood of $S^2\times S^1$.

The general case similarly follows using Theorem \ref{higher genus}.  \end{proof}

\begin{remark} By induction Marumoto \cite{Ma} proved  more generally that two locally flat PL $m$-discs in an $n$-sphere, $n>m$ with the same boundary are topologically isotopic rel boundary.  Here is an outline of his argument for smooth discs in the $n$-sphere for the representative case $m=2, n=4$, where we use \cite{Ce1}, \cite{Pa} to avoid his induction steps. Actually, the below argument works in all dimensions and codimensions since the same is true of \cite{Ce1}, \cite{Pa} and the Alexander isotopy.

Start with $D_0, D_1$ where $D_1$ is the standard 2-disc in $S^4$ and $\partial D_0= \partial D_1$.  Then by \cite{Ce1}, \cite{Pa} there is a diffeomorphism $f:S^4\to S^4$ taking $D_0$ to $D_1$ fixing $\partial D_0$.  We can assume that $f$ fixes pointwise a neighborhood of $\partial D_0$.  Next remove a small ball about a point in $\partial D_0$.  After restricting and reparametrizing we obtain a map $g:B^4\to B^4$ such that $g(E_0)=E_1$ where the $E_i$'s are the restricted reparametrized $D_i$'s.  Here $B^4$ is the unit ball in $\BR^4, \partial E_0$ is a straight properly embedded arc connecting antipodal points of $\partial B^4$ and $g|\partial B^4=\id$.  Finally apply the Alexander isotopy to obtain  a topological isotopy of $g$ to the identity  which fixes $\partial E_0$ pointwise.  \qed\end{remark}  

More generally we have the following uniqueness of spanning discs in simply connected 4-manifolds.

\begin{theorem} \label{uniqueness simple} If $D_0$ and $D_1$ are smooth embedded discs in the simply connected 4-manifold $M$ such that $\partial D_0 = \partial D_1=\gamma$, then there exists a smooth isotopy of $M$ taking $D_0$ to $D_1$ fixing $\gamma$ pointwise if and only if the mapped sphere $S=D_0\cup_{\gamma} D_1$ is inessential in $M$.\end{theorem}

\begin{proof}  If $D_0$ and $D_1$ are isotopic, then the isotopy sweeps out a contracting ball for $S$.  Conversely, after an initial isotopy of $D_1$ we can assume that it coincides with $D_0$ near $\gamma$ and that the interior of the mapped 3-ball $B$ defining the contraction of $S$ intersects $\gamma$ algebraically zero.  Indeed,  the second isotopy in the proof of Theorem \ref{marumoto} enables modification of the intersection number.  These intersections can be eliminated using immersed Whitney discs.  Next surger  $\gamma$ to obtain the simply connected manifold $N$ so that $D_0$ and $D_1$ give rise to homotopic spheres $R_0$ and $R_1$  with common transverse sphere $G$,  that  coincide near their intersection with $G$.  By Theorem \ref{main}, $R_0$ and $R_1$ are isotopic via an isotopy fixing $G$ pointwise and hence $D_0$ and $D_1$ are isotopic rel boundary.\end{proof}

\begin{remark} In a similar manner, using Theorems \ref{multi main} and \ref{genus}, one can obtain uniqueness theorems for certain surfaces spanning simple closed curves in closed 4-manifolds with no 2-torsion in their fundamental groups.\end{remark}

One can ask the following parametrized form in the smooth category.

\begin{question}  For $i=1,2$ let $f_i:D^k\to S^4$ be smooth embeddings such that $f_1|\partial D^k=f_2|\partial D^k$.  Is there a smooth isotopy $F:S^4\times I\to S^4$ such that $F_0=\id_{S^4}, F_t(f_1(x))=f_1(x)$ for $x\in \partial D^k$ and $t\in [0,1]$ and for $y\in D^k, F_1(f_2(y))=f_1(y)$?\end{question}

\begin{remark} For $k\le 3$ the unparametrized version implies the parametrized one by \cite{Ce3} for k=3 and  \cite{Sm3} for $k=2$  with the $k=1$ case being  elementary.   The point of this question is to link various theorems, conjectures and questions.
\vskip 8pt
\noindent\emph{Case k=1:}  This is the theorem \emph{homotopy implies isotopy for curves in 4-manifolds}.
\vskip 8pt
\noindent\emph{Case k=2:}  This is Theorem \ref{marumoto}.
\vskip 8pt
\noindent\emph{Case k=3:}  This implies the Schoenflies conjecture.  Indeed the Schoenflies conjecture is equivalent to a positive resolution of the question after allowing lifting of the $f_i$'s to some finite branched covering of $S^4$ over $ \partial (f_i(D^3))$.  See \cite{Ga}.
\vskip 8pt
\noindent\emph{Case k=4:}  This is the question of connectivity of  $\Diff_0(B^4, \partial)$.\end{remark}





\begin{question} Does Theorem \ref{multi main} hold without the $G$-inessential condition?  What if $G$-inessential is replaced by $\pi_1$-inessential?\end{question}

The following are special cases of the long standing questions of whether a sphere $R$ in $\cptwo$ homologous to $\cpone$ is equivalent up to isotopy or diffeomorphism to the standard $\cpone$.  See problem 4.23 \cite{Ki}.

\begin{questions}  \label{melvin} i) If $R$ is a smooth sphere in $\cptwo$ that intersects $\cpone$ once is $R$ isotopically standard?

ii) \cite{Me} Is $(\cptwo, R)$ diffeomorphic to $(\cptwo, \cpone)$?\end{questions}

\begin{remark}  In his unpublished 1977 thesis, Paul Melvin \cite{Me} showed that blowing down $\cptwo$ along $\cpone$ transforms $R$ to a 2-knot $T$ in $S^4$ and  Gluck twisting $S^4$ along $T $ yields $S^4$ if and only if $(\cptwo,R)$ is diffeomorphic to $(\cptwo, \cpone)$.  He gave a positive answer to ii) for 0-concordant knots.\end{remark}

\enddocument